\documentclass[]{article}      
\usepackage{graphicx}
\usepackage{authblk}

\usepackage[latin1]{inputenc}
\usepackage[T1]{fontenc}
\RequirePackage[colorlinks, citecolor=blue,urlcolor=blue]{hyperref}

\usepackage{latexsym,enumerate}
\usepackage{amsmath,amsopn,amstext,amscd,amsfonts,amssymb,amsthm,pdfsync,amsbsy}
\usepackage{mathrsfs}
\usepackage{url}
\usepackage{graphicx, float}
\usepackage[usenames, dvipsnames, svgnames, table]{xcolor}
\usepackage{natbib}

\newtheorem{proposition}{Proposition}
\newtheorem{lemma}{Lemma}
\newtheorem{theorem}{Theorem}
\newtheorem{corollary}{Corollary}
\theoremstyle{remark}
\newtheorem{rem}{Remark}

\let\epsilon=\varepsilon
\let\eps=\epsilon
\let\phi=\varphi

\DeclareMathOperator*{\argmin}{\arg\min}
 
\newcommand{\field}[1]{\mathbb{#1}}
\newcommand{\R}{\field{R}}

\newcommand{\N}{\field{N}}
 
\newcommand{\F}{{\mathscr{F}}}

\newcommand{\Nr}{{\mathcal N}}

\newcommand\eref[1]{(\ref{#1})}

\def\I{{\mathbf 1}}

\begin{document}
	\title{Convergence rates for smooth $k$-means change-point detection}
	
\author{Aurélie Fischer $\&$ Dominique Picard}
\affil{\large LPSM (UMR CNRS 8001)\\\large Universit\'e Paris Diderot\\ \large Case courrier 7012 \\\large 75205 Paris Cedex 13, France\\\medskip\large \url{aurelie.fischer@univ-paris-diderot.fr}\\\large \url{dominique.picard@univ-paris-diderot.fr}
}

\maketitle
			
			\begin{abstract}
					In this paper, we consider the estimation of a change-point for possibly high-dimensional data in a Gaussian model, using a $k$-means method. 
				We prove that, up to a logarithmic term,  this change-point estimator  has a minimax rate of convergence. 
				
				Then, considering the case of sparse data, with a Sobolev regularity, we propose a smoothing procedure based on Lepski's method and show that the resulting estimator attains the optimal rate of convergence.
				
				Our results are illustrated by some simulations. As the theoretical statement relying on Lepski's method depends on some unknown constant, practical strategies are suggested to perform an optimal smoothing. 
			\end{abstract}
			


\section{Introduction} \label{sec-intro}

\subsection{Clustering and change-point detection}

An important problem in the vast domain of statistical learning is the question of unsupervised classification of high dimensional data.
Many examples fall into this category such as the classification of curves, of images, the segmentation of  domains (geographical, economical, medical, astrophysical...) into homogeneous regions given  observed quantities.

The problem can be summarized into finding $r$ clusters (``coherent ensembles'') for  $n$ individuals, given that each one is observed through a $d$-dimensional vector. We will be mostly interested here (although it is not a necessary condition) in the case where $d$ is large, possibly very large compared to $n$.
A lot of examples fall into this setting, for instance the case where you need  to partition a geographic zone into smaller areas which are highly homogeneous with respect to climatic quantities such as temperature, wind measures. In such a case the data generally consists, for $n$ meteorological stations ---regularly spaced on the zone--- of years of hourly measures, leading to a huge vector of dimension $d$ for each station.

In such a case, the clustering problem can generally be summarized in two steps: in a first step, a preprocessing finds a representation of the data, which can be the raw data or a more subtle transformation.
Then, a segmentation  algorithm is performed on the transformed data. We will focus here on the case where this algorithm is the famous $k$-means algorithm, corresponding to estimation via the empirical risk minimizer.

Between or around these two steps and especially when $d$ is very large, there is a need for ``smoothing'', or in other terms, to reduce the dimension $d$. This is especially important from a computational point of view. Without this step, the $k$-means algorithm might be unstable or even not work at all.

In this paper, we will consider the problem from a theoretical point of view (as opposed to algorithmic point of view).

More precisely, we will concentrate on the following questions:
\begin{itemize}
	\item[(1)] Without referring to the feasibility, what is more efficient to obtain a clustering result: keeping the raw data or smoothing the data?
\item [(2)] If the data are high-dimensional but ``sparse'', is there a way to use this sparsity to get better clustering results?
\item [(3)] If smoothing proves to be efficient, how could it be performed? Do usual nonparametric smoothing methods work as well in a clustering problem?
\item [(4)]  Does on-line (signal by signal smoothing ---station by station in the meteorological example) performs as well as  off-line smoothing (using a preprocessing  involving all the signals)?

\end{itemize}

We will attack this problem in a much simpler setting, and see that even in this simplified framework, there are still open questions. The first simplification will be that the number of classes will be fixed: we will assume that there are exactly two classes. Moreover, we make the more restrictive assumption that the change between one class and the other occurs on a time scale. In other words, there exists a change-point $\tau$: before $n\tau$, the observations have a certain regime, after $n\tau$, they have another regime.

In a wide range of areas,  change-point problems may occur in a high-dimensional context. This is the case for instance in the analysis of  network traffic data \cite{LR09,LLC12}, in bioinformatics, when studying copy-number variation \cite{BlV11,ZSJL}, in astrostatistics \cite{BMS11,SMF14,MCMBPB16} or  in multimedia indexation \cite{HVLC09}. In these practical applications, the number of observations is relatively small compared with their dimension, with the change-point possibly occurring  only for  a few components.

The change-point problem has its own interest and has also a long history. For an introduction to the domain, the reader may refer for instance to the monographs and articles by \cite{Shi78}, \cite{Rit90}, \cite{Mu92}, \cite{BassNiki}, \cite{BD93}, \cite{CMS}, \cite{CH97}.
Change-point detection based on resampling has been investigated in \cite{sample1} and \cite{sample2}.

Minimax estimation is considered for example in \cite{Kor87}, in the  Gaussian white noise model. 
In this framework, high-dimensional change-point problems are studied by \cite{Korostelev2008}, who propose an asymptotically minimax estimator of the change-point location, when the Euclidean norm of the gap tends to  $\infty$ as $d \to\infty$.

\subsection{Main results and organization of the paper}
We begin Section 2 by introducing the two class model, and the change-point model.
As well, we present the empirical minimizer estimator of the change-point, depending on the smoothing.

We prove that up to a logarithmic term the empirical minimizer ($k$-means) has a minimax rate of convergence. It is important to notice that we do not know whether this logarithmic term is necessary or not. Indeed, in \cite{Korostelev2008}, ``the edges are known'', meaning that the minimax rate is established in the case where the change-point cannot  occur before a {\textit{known}} proportion of the observation and as well after a {\textit{known}}  proportion of observation. Our method of estimation is agnostic to this knowledge, creating obvious additional  difficulties.

Secondly, in Section 3, we show that if the data is sparse ---here, in a ``Sobolev'' sense, there exists an optimal smoothing.
To attain this optimal smoothing, we provide a method relying on the Lepski smoother. This method, which basically consists in performing a Lepski smoothing on a surrogate data vector built on the whole observation, has the advantage on being performed beforehand and will not create additional computational difficulties in the $k$-means algorithm.
It could be interestingly compared with the lasso-$k$-means (see \cite{levrard2013, levrard2015}), which is known to be difficult to implement in large data sets.

We provide in Section 4 a numerical experimentation  of our methods, which shows promising results.

Section 5 is devoted to the proofs.

\section{Two class model -- Change-point model}

Let $n\geq3$. For a set $A$, we will denote  its cardinal by $\# A$.
We observe $n$ independent signals $Y_1,\ldots,Y_n$. We assume that each signal $Y_i$, $i=1,\dots,n,$ is observed discretely, through $d$ components: for every $i$, $Y_i=(Y_{i,1},\ldots, Y_{i,d})$ is a random vector with values in $\R^d$.

We consider the following  Gaussian clustering framework.
We suppose that there exist a set $A\subset\; \{1,\ldots,n\}$ (unknown)  and two  vectors $\theta^-$ and $\theta^+$ of $\R^d$  such that
\begin{align*}
Y_i&=\theta_i+\eta_{i}, \quad 1\le i\le n, \quad \eta_{i} \; \mbox{i.i.d.}\;\Nr(0, \sigma^2I_d),\\
\theta_i&=\theta^- ,\; \forall i\in A,\\
\theta_i&=\theta^+ ,\; \forall i\in A^c.
\end{align*}

Recently, \cite{EniHar17} have considered a similar Gaussian model from the point of view of testing. The aim is to test whether there is a change-point or not. The Gaussian vectors may be high-dimensional, with the change possibly occurring in an unknown subset of the components. In a double asymptotic setting, where the number of observations and the dimension grow to infinity, the authors build an optimal test.

\begin{rem}
	\begin{enumerate}
		 \item The covariance matrix of the noise $\eta_i$ is chosen to be  proportional to identity. Choosing instead a covariance of the form $\sigma^2J$, where $J $ is a known matrix, would lead to a somewhat identical discussion by a simple change of variable, provided that we make also the appropriate regularity assumptions on the parameters $\theta^+$ and $\theta^-$. 
	
	\item For a first step, we suppose here $\sigma^2$ to be known. Note that $\sigma^2$ may depend on $d$. For instance, if we think of a Gaussian white noise, then $\sigma^2$ would frequently be of the form $\sigma_0^2/d$, where $\sigma_0^2$ is an absolute known constant. We will not investigate the case where $\sigma^2$ is unknown.
	
	\item
	The fact that the noise is Gaussian is a simplification which is useful but not essential: basically, concentration inequalities are needed and similar results are likely under sub-Gaussian hypotheses on the errors.
	\end{enumerate}
	
\end{rem}

We will mainly be interested in the behavior of the maximum likelihood  (MLE) estimators (also  known as two-class $k$-means estimators in this context):

\begin{multline*}
\hat B= 
\argmin_{B\subset \{1,\ldots,n\}}
\left\{
\sum_{i\in B}\sum_{\ell=1}^d\Big(Y_{i,\ell}-\frac 1{\#B}\sum_{i\in B}Y_{i,\ell}\Big)^2\right.\\
\left.+\sum_{i\in B^c}\sum_{\ell=1}^d\Big(Y_{i,\ell}-\frac 1{\#B^c}\sum_{i\in B^c}Y_{i,\ell}\Big)^2\right\}.
\end{multline*}

\subsection{Simplified two-class  model: change-point  clustering}
Change-point clustering essentially means that the set $A$ has the following form: 
$$A=\{1,\ldots ,[n\tau]\}.$$
In other terms, 
there exist a change-point $ 0<\tau<1$ and two  vectors $\theta^-$ and $\theta^+$ of $\R^d$, such that the model is given by \begin{align}\label{eq:model}
&Y_i=\theta_i+\eta_{i}, \quad 1\le i\le n, \quad \eta_{i} \; \mbox{i.i.d.}\;\Nr(0, \sigma^2I_d),\\
&\forall i\le { n\tau},\quad \theta_i=\theta^-\nonumber, \\
&\forall i> { n\tau}, \quad \theta_i=\theta^+\nonumber.
\end{align}
To prove our results, we will additionally  assume the following conditions in different places.

\subsubsection{Condition [Edge-out]} We  assume that there exists $0<\eps<1/2$, such that $\eps<\tau<1-\eps$. This condition is introduced basically to avoid problems at the border of the interval $[0,1]$.
It is important to notice that our results will depend on $\eps$. However, the procedure is agnostic to $\eps$, which  is not supposed to be known.

\subsubsection{  Condition on the means [$\Theta(s,L)$]}
 For $s>0$,
we define
$$ \Theta(s,L):=\left\{\theta\in \R^d, \; \sup_{K\in \N^*}K^{2s}\sum_{k\ge K}(\theta_k)^2\le L^2\right\}.$$
\noindent We will suppose that $\theta^-$ and $\theta^+$ are in $\Theta(s,L)$.

 \begin{rem}
 	This assumption expresses a form of sparsity of the coefficients, which reflects an ordering in their  importance: the first ones are supposedly more important than the last ones. This is quite a reasonable assumption when, as it is generally the case,  the clustering is operated in two steps: during the first one,  the data is projected to a feature space (in a linear or nonlinear way)  supposedly reflecting the salient parts of the data. 

Note that there are possible extensions to other kinds of sparsity, considering for instance coefficients belonging to the set $${ \Theta_q(L):=\left\{\theta\in \R^d, \; \sum_{k}|\theta_k|^q\le L\right\}},$$ where $q<1$, but this choice requires more sophisticated smoothing algorithms.
 \end{rem}

\subsection{Smooth estimation of $\tau$}
Our problem is, considering the regularity assumptions, to determine, for the problem of estimating the change-point $\tau$,  whether or not it is efficient to smooth the data. More specifically, we will investigate the effect of  replacing the vectors $Y_i= (Y_{i,1}, \dots, Y_{i,d})$, $i\le n$ (called in the sequel ``raw data''), by, for $T<d$,  $Y_i(T):=(Y_{i,1}, \dots, Y_{i,T})$, $i\le n$, the vectors of $\mathbb{R}^T$ of the $T$ first coordinates of 
$Y_i$.

We consider the associated empirical minimizer:
\begin{multline*}
\hat k(T)= 
\argmin_{k\in \{2,\ldots ,n-2\} }
\left\{
\sum_{i=1 }^{k}\sum_{j=1}^T\Big(Y_{i,j}-\frac 1{k}\sum_{i=1}^ {k}Y_{i,j}\Big)^2\right.\\ \left.+
\sum_{i=k+1}^n\sum_{j=1 }^T\Big(Y_{i,j}-\frac 1{n-k}\sum_{i= k+1}^n Y_{i,j}\Big)^2\right\},
\end{multline*}
and we set $$\hat \tau(T)=\frac{\hat k(T)}{n}.$$

To investigate the behavior of the procedure, we begin by proving a result  describing the behavior of this estimated change-point $\hat\tau(T)$.

In the sequel, we will use the notation 
\begin{equation*}
\Delta^2:= \sum_{j=1}^d(\theta^{-}_j-\theta^{+}_j)^2=\|\theta^+-\theta^-\|^2.
\end{equation*}
We also  define, for $T\leq d$,
\begin{equation*}
\Delta^2_T:=\sum_{j=1}^T(\theta^{-}_j-\theta^{+}_j)^2, \quad \Psi_n(T,\Delta_T)=\frac {\sigma^2}{n\Delta^2_T}\left(1\vee \frac {\sigma^2T}{n\Delta^2_T}\right).
\end{equation*}
\begin{proposition}\label{beta}Let us assume condition [edge-out]. 

For any $\gamma >0$, there exist  constants $\kappa(\gamma,\eps)$ and $c(\gamma,\eps)$ such that, 
if
$$
{\Delta^2_T}\ge c(\gamma,\eps)\frac{\sigma^2\ln (n)}n,$$   
then 
$$P\Big(|\hat \tau(T)-\tau|\ge  \kappa(\gamma,\eps) \ln(n)\Psi_n(T,\Delta_T)\Big)\le n^{-\gamma}.$$
 \end{proposition}

\begin{rem}\begin{enumerate}
		\item 
	 Note that no condition on the sparsity of $\theta^+$ and $\theta^-$ is needed for this result.

\item\label{mM}
 Thanks to \cite{Korostelev2008}, one can observe that $\Psi_n(T,\Delta_T)$ is the minimax rate in this framework. Compared to their result, we are apparently loosing a logarithmic factor. But it is important to stress that in their paper, the bound $\eps$ was supposed to be known, whereas our estimator $\hat\tau(T)$ is adaptive in $\eps$. So it is not absurd to suggest that the logarithmic term might be necessary.
  
  \item For $T=d$, $\Psi_n(d,\Delta_d)=\Psi_n(d,\Delta)$.
  The rate is composed of two different regimes: a ``good one'' ${\frac {\sigma^2\ln(n)}{n\Delta^2}}$, not depending on the dimension $d$ and a ``slow one'' $ \frac {\sigma^4\ln(n)d}{(n\Delta^2)^2}$ which is rapidly deteriorating with the dimension.
   
From the results above, we deduce that if $c(\gamma,\eps)\frac{\sigma^2\ln (n)}n \le \Delta^2< \frac{\sigma^2d}n$, the rate of convergence is $ \frac {\sigma^4\ln(n)d}{(n\Delta^2)^2}$, whereas 
 if $ \Delta^2\ge \frac{\sigma^2d}n\vee  c(\gamma,\eps)\frac{\sigma^2\ln (n)}n$, it is ${\frac {\sigma^2\ln(n)}{n\Delta^2}}$. This last rate is obviously much better, and with this latter condition on $\Delta$, taking $T=d$ (so raw data) allows to obtain the best  rate ${\frac {\sigma^2\ln(n)}{n\Delta^2}}$. Taking a smaller $T$ could lead to a reduction of $\Delta_T$ damaging the rate.

 However this latter condition is quite restrictive on $\Delta$ when $d$ is large.
 In the next paragraph, we will try to refine this condition, gaining on the size $T$ of the smoothed vector.

\item Without assumptions on the behavior of the parameters $\theta^+$ and $\theta^-$, there is nothing much to hope about the way $\Delta_T$ is increasing in $T$. However, the regularity assumption $[\Theta(s,L)]$ allows us to assume that for $T$ such that $\Delta^2\ge 8L^2 T^{-2s}$, then $\Delta_T$ and $\Delta$ are comparable, in the sense that $\Delta_T^2\ge\Delta^2/2$. Indeed, 
 $\Delta^2 -\Delta^2_T=\sum_{j=T+1}^{d}(\theta^{-}_j-\theta^{+}_j)^2\leq4 T^{-2s}L^2$, so that  $$\frac{\Delta^2_T}{\Delta^2}\geq 1-\frac{4 T^{-2s}L^2}{\Delta^2}\geq 1/2.$$ This is precisely what is exploited in the first part of Theorem \ref{T,s} below.

\item Let us observe that if $\Delta_T$ and $\Delta$ are comparable, then 
 $\Psi_n(T,\Delta_T)\sim \Psi_n(T,\Delta)$ is much easier to analyse.
 In particular we see that again it is composed of two regimes ---a slow one and a good one--- and the dependence in $T$ is more clear: $\frac{\sigma^2\ln(n)}{n\Delta^2}$ for $T\le \frac{n\Delta^2}{\sigma^2}$, and $\frac{\sigma^4\ln(n)T}{(n\Delta^2)^2}$ for larger $T$'s.
 \\
 This is also corresponding to what is frequently observed in practical applications: when the dimension is increasing, one first observes indications of convergence being decreasing, then stable for a while and then increasing substantially.
	\end{enumerate}
\end{rem}

\subsection{Consequences of Proposition \ref{beta}}

An immediate consequence of Proposition \ref{beta} is the following theorem.

\begin{theorem} \label{T,s}

We consider the model \eref{eq:model}, and we assume conditions [edge-out],
and $[\Theta(s,L)]$. 

For any $\gamma >0$, there exist  constants $\kappa(\gamma,\eps)$ and $c(\gamma,\eps)$ such that, 
if
$$
\Delta^2\ge \left[2c(\gamma,\eps)\frac{\sigma^2\ln (n)}{n}\vee 8L^2 T^{-2s}\right],$$  
then 
$$P\Big(|\hat \tau(T)-\tau|\ge \kappa(\gamma,\eps)\ln(n)\Psi_n(T,\Delta)\Big)\le n^{-\gamma}.$$
If, now,
\begin{equation}
\Delta^2\ge \left[2c(\gamma,\eps)\frac{\sigma^2\ln (n)}{n}\vee 8L^2 T^{-2s}\vee\frac{\sigma^2T}{n} \right],\quad \lambda \geq \kappa(\gamma,\eps) \ln(n),\label{cond-delta}
\end{equation} 
$$P\Big(|\hat \tau(T)-\tau|\ge \kappa(\gamma,\eps)\frac{\sigma^2\ln(n)}{n\Delta^2}\Big)\le n^{-\gamma}.$$

\end{theorem}
Optimizing Condition \eref{cond-delta} in $T$   leads to { 
 $T_{opt}\sim T_s:= \left(\dfrac {8L^2n}{\sigma^2}\right)^{\frac1{1+2s}}$}.
\begin{corollary}\label{cor:1}
Under the conditions above, 
for any $\gamma >0$, there exist  constants $\kappa(\gamma,\eps)$ and $c(\gamma,\eps)$ such that, if 
\begin{equation*} 
\Delta^2\ge \left[2c(\gamma,\eps)\frac{\sigma^2\ln (n)}{n}\vee\left(\frac{\sigma^2}n\right)^{\frac{2s}{1+2s}}\left(8L^2\right)^{\frac1{1+2s}} \right],\label{cond-delta1}
\end{equation*} 
$$P\Big(|\hat \tau(T_s)-\tau|\ge \kappa(\gamma,\eps)\frac{\sigma^2\ln(n)}{n\Delta^2}\Big)\le n^{-\gamma}.$$

\end{corollary}

\begin{rem}

\begin{enumerate}
	\item We see here that there is an advantage in smoothing since it allows to obtain the best rate with less restricting conditions on the gap $\Delta$.

\item  We see that the greater the parameter $\Delta^2 $, the faster the rate of convergence of $\hat \tau$, which is natural, since $\Delta^2 $ corresponds to the Euclidean distance between the two means $\theta^+$ and $\theta^-$ and the segmentation task is obviously easier when groups are well-separated.

\item  At first sight, the rate of convergence and the  conditions could seem quite unsatisfactory, but observe that very often 
$\sigma^2 $ is of the form $\frac{\sigma_0^2}d$.
 In this case,
 the rate of convergence is of the order $$ {\left(\frac {nd}{\sigma_0^2}\right)^{\frac{-2s}{1+2s}}\Delta^{-2}}.
$$

\item  If we now look for a procedure searching for an optimal $T$ in an adaptive way (without knowing the regularity $s$), some remarks can be made before giving a solution. In particular, one may ask whether it is possible to optimize individually (on each signal $Y_j$ of $\R^d$), or if it is  necessary to perform an off-line preprocessing (requiring the use of all the signals). The form of the optimal smoothing 
$T_s\sim\left(\frac {nd}{\sigma_0^2}\right)^{\frac1{1+2s}}$ allows to answer this question, 
proving that any adaptive smoothing performed individually on each signal $Y_j$ (thresholding, lasso... ) 
 would lead instead to an optimal smoother of the form:
${T_{opt}= \left(\frac {d}{\sigma_0^2}\right)^{\frac1{1+2s}}}$,
inevitably creating  in the rates a  loss of a polynomial factor in $n$.
This  means that it is certainly more efficient to find a procedure performing the  smoothing  globally (off-line).
\end{enumerate}

\end{rem}

\section{Adaptive choice of $T$}

\subsection{Lepski's procedure} To begin with, let us recall the classical Lepski procedure (see \cite{Lep1, Lep2, Lep3}). 
In the standard Gaussian white noise model, 
\begin{equation}
Z_j= \beta_j+\eps_j,\; j=1,\ldots,d, \label{pseudo}
\end{equation} where the $\eps_j$'s are i.i.d. $\Nr(0, \nu^2)$,
a standard choice for the smoothing parameter $T$ consists in defining $\hat{T}$ as follows:

\begin{equation*}
\hat T:= \min\left\{k\geq 1:  \forall d  \ge j\ge m\ge k, \sum_{\ell=m}^j (Z^\ell)^2\le C_{\mathcal L}j\nu^2\ln (d\vee n) \right\},\label{Lepski}
\end{equation*}
where $C_{\mathcal L}$ is a tuning constant of the procedure.

\subsection{Preprocessing}
Here, we will use Lepski's procedure in a special case. 

First, using the complete data set (so off-line), we will create a surrogate data vector, estimating a parameter $\beta$ of regularity $s$. These data will be used to find an optimal $\hat T$.

Of course, it is known  that estimating the regularity of a signal is impossible without important extraneous assumptions, but what adaptive procedures are producing ---and especially in this case Lepski's procedure--- is a smoothing parameter $\hat T$ which, with overwhelming probability will be  smaller than the optimal one $T_s$ (defined above). This is not enough when one wants to estimate the regularity $s$ (unless extraneous assumptions are imposed). However, fortunately,  Lepski's procedure  also controls the bias of the procedure,  assuring that $\Delta^2-\Delta_{\hat T}$ is still reasonable, which is precisely the need here.

\subsubsection{Surrogate data}
For the sake of simplicity, we consider that $n$ is even; otherwise, the modifications are elementary.

Let us consider the following vector:
$$Z_j= \frac 1n \sum_{i=1}^nY_{i,j}- \frac 2n \sum_{i=1}^{n/2}Y_{i,j},\quad j=1,\ldots, d.$$

It is easy to see that this model is a special case of \eref{pseudo}
with 
\begin{align*}
&\beta_j= (1-\tau)(\theta_j^+-\theta_j^-)\I_{\{\tau\ge 1/2\}}+ \tau(\theta_j^+-\theta_j^-) \I_{\{\tau< 1/2\}}, \\ &\eps_j=\sum_{i=1}^{n/2}\frac{-1}n\eta_{i,j}+\sum_{i=n/2+1}^n \frac{1}n\eta_{i,j}, \\ &\nu^2= \frac{\sigma^2}n.
\end{align*}

We consider the Lepski procedure applied to the vector $Z$, producing a smoothing parameter $\hat T$. This smoothing parameter is then just plugged in the $k$-means procedure for estimating $\hat \tau$.

The following theorem states that the method leads to an optimal selection, up to logarithmic terms  (same convergence rate as in Corollary \ref{cor:1} with $T_s$).

\begin{theorem}\label{hatT}

In the model \eref{eq:model},  we assume that  $\theta^+$ and $\theta^-$  belong to  $\Theta(s,L)$. We suppose that there exists a constant $\alpha>0$ such that
 $$\frac n {\sigma^{2}}\ge \alpha \ln d.$$
We set 
$$\hat T:= \min\left\{k\geq 1:  \forall d  \ge j\ge m\ge k, \sum_{\ell=m}^j (Z^\ell)^2\le C_{\mathcal L}j\frac{\sigma^2}{n}\ln (d\vee n) \right\}.$$
Then, for any $\gamma >0$, there exist  constants $\kappa(\gamma,\eps)$ and $c(\gamma,\eps)$ and $R$ such that, 
if
$$ 
\Delta^2\ge 2c(\gamma,\eps)\frac{\sigma^2\ln (n)}{n}\vee  R\left(\frac{\sigma^2\ln (d\vee n)}{n}\right)^{\frac{2s}{1+2s}},$$  
then 
$$P\Big(|\hat \tau(\hat T)-\tau|\ge \kappa(\gamma,\eps)\frac{\sigma^2\ln(n)}{n\Delta^2}\Big)\le n^{-\gamma}.$$

\end{theorem}

\section{Numerical study}
In this section, we provide  some simulations illustrating our theoretical results. 
\subsection{Rate of convergence}

In this experiment, we study the rate of convergence of the estimator $\hat{\tau}$. Let $d=20$, $T=10$, $\sigma=1$, $\tau=0.3$.
Let us consider data generated from Model \eqref{eq:model} with the means $\theta^-$ and $\theta^+$ obtained from the following distribution: $\theta^-\sim\Nr(0,\frac{1}{20j^2}))$
$\theta^+\sim\Nr(-\theta^-,10^{-4})$.

To get a first insight about the rate of convergence, we simulate 1000 times a sample of length $n$, for $n$ chosen between 20 and 4000, and plot in Figure \ref{fig:vit} the mean and median of the error $|\tau-\hat \tau|$ over the 1000 trials in function of $n$,  together with the function $n\mapsto \ln(n)\Psi_n(T,\Delta_T)$ corresponding to the theoretical rate of convergence  obtained in Proposition~\ref{beta}. Note that the rate of convergence of  $|\tau-\hat \tau|$ is given in the proposition up to a constant $\kappa(\gamma,\eps)$. Nevertheless,  the figure provides an appropriate illustration of the result as soon as $n$ is large enough.

\begin{figure}[H]
	\includegraphics[width=0.49\textwidth]{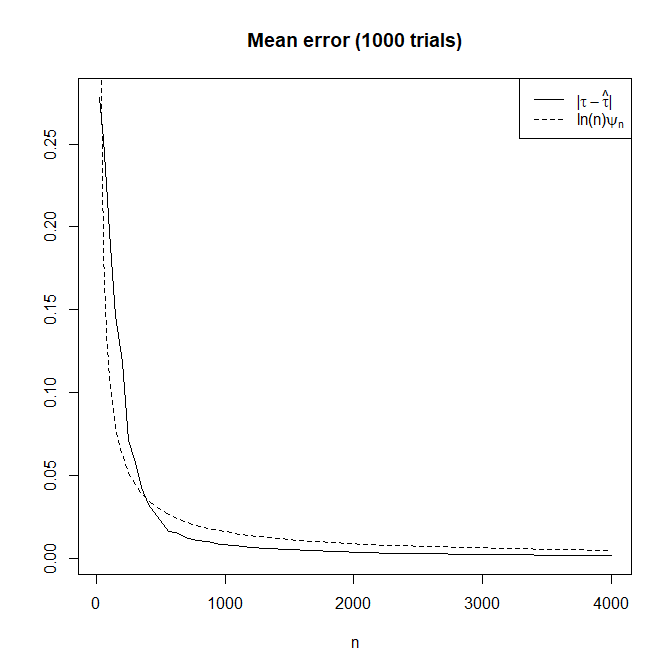}
	\includegraphics[width=0.49\textwidth]{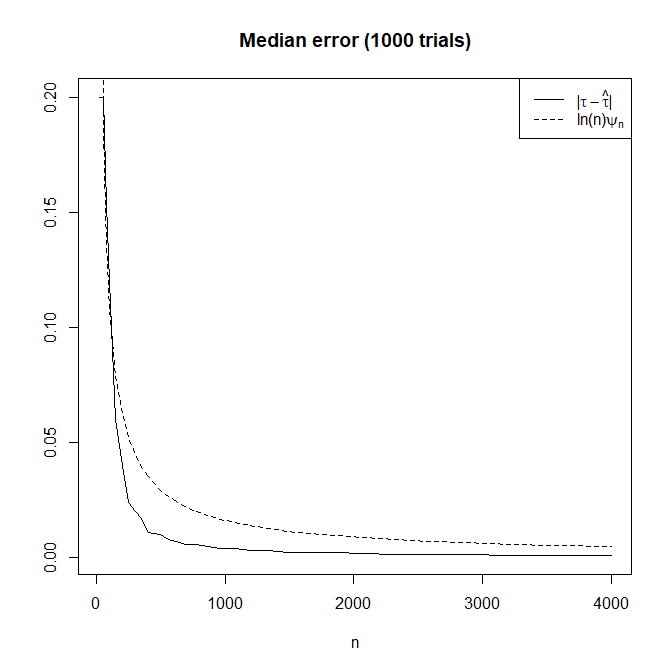}
	\caption{Plot of $|\hat\tau-\tau|$ as a function of $n$ (mean over 1000 trials).}\label{fig:vit}
\end{figure}

\begin{figure}[H]
	\includegraphics[width=0.49\textwidth]{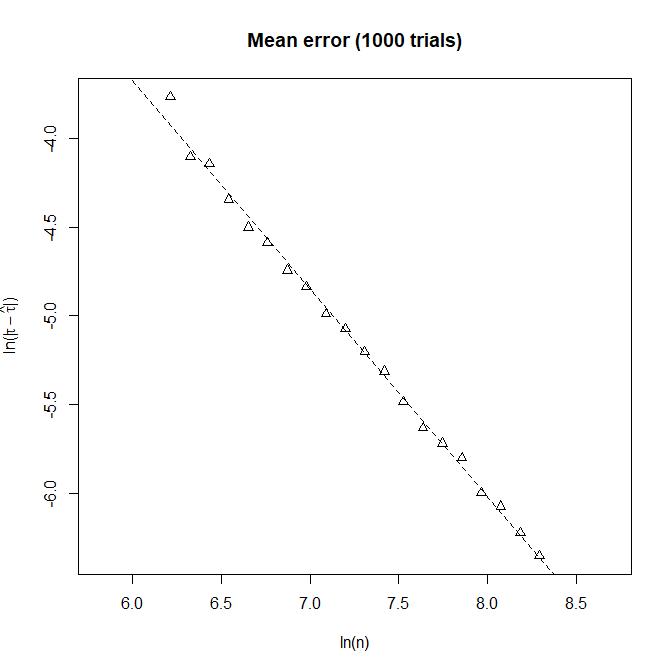}
	\includegraphics[width=0.49\textwidth]{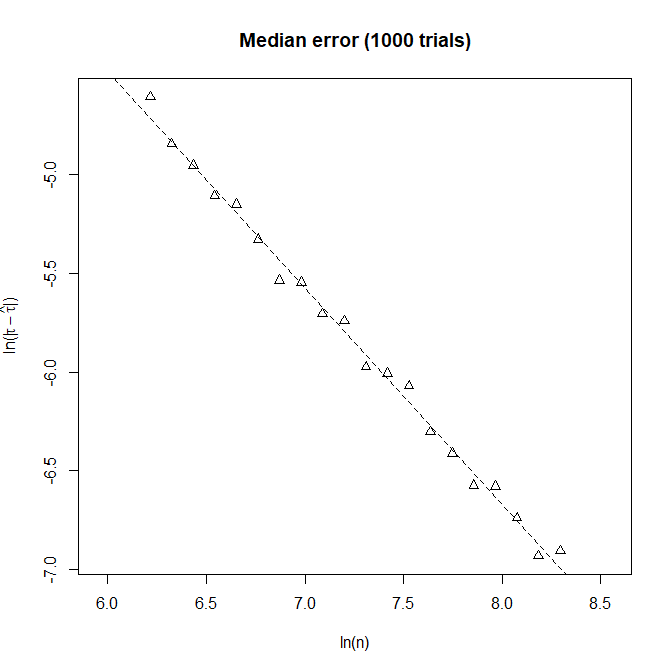}
	\caption{Plot of $\ln(|\hat\tau-\tau|)$ as a function of $\ln(n)$ (mean and median over 1000 trials).}\label{fig:vitreg}
\end{figure}

Then, simulating 1000 samples, for each value of the sample size $n$  between 500 and 4000, we try to estimate of the rate of convergence by computing the linear regression of $|\tau-\hat \tau|$ by $\ln(n)$: omitting the logarithmic factor, an exponent $-1$ is to be found, corresponding to the rate of convergence  $\frac 1n$.
Figure \ref{fig:vitreg} provides an illustration of this linear regression, considering again the mean and the median over the 1000 trials. On this example, the estimated slope of the regression line is $-1.172$ for the mean and $-1.098$ for the median.

\subsection{Selection of $T$}

In Theorem \ref{hatT}, we suggest to select $T$ using the Lepski method.  Before introducing a practical procedure for the selection of $T$, let us illustrate the fact that the performance of the estimator $\hat{\tau}$ may indeed vary a lot as a function of $T$, so that selecting the right $T$ is a crucial issue in the estimation of $\tau$.

We set $d=200$, $n=100$, $\sigma=1$, $\tau=0.3$.
We consider  data generated from Model \eqref{eq:model} with means $\theta^-$ and $\theta^+$ built as follows:
\begin{itemize}
	\item Case $A$: $\theta^-\sim\Nr(0,V)$, $\theta^+\sim\Nr(0,V)$, $V=\mbox{diag}(v_1,\dots,v_d),$ $v_j= \frac{1}{2j^2}$ for $j=1,\dots,d$.
	\item Case $B$: $\theta^-$ is such that $\theta^-_j\sim\Nr(0,1/2)$ for $j=1,\dots,20$, and $\theta^-_j\sim\Nr(0,\frac{1}{2(j-20)^2})$ for $j=21,\dots,d$.
	$\theta^+$ is such that $\theta^+_j~\Nr(\theta^-_j,10^{-2}))$ for  $j=1,\dots,20$, and $\theta^+_j\sim\Nr(0,\frac{1}{2(j-20)^2})$ for $j=21,\dots,d$.
\end{itemize}

We simulated 5000 data sets according to Model \eqref{eq:model} in  each of the two cases. Figure \ref{fig:fctTA} and \ref{fig:fctTB} show the mean and median error $|\hat\tau-\tau|$ over the 5000 trials in function of $T$. In the first case, the best result is obtained already with $T=1$, whereas for the second, taking $T$ around 30 is a good choice.
\begin{figure}[H]
	\includegraphics[width=0.49\textwidth]{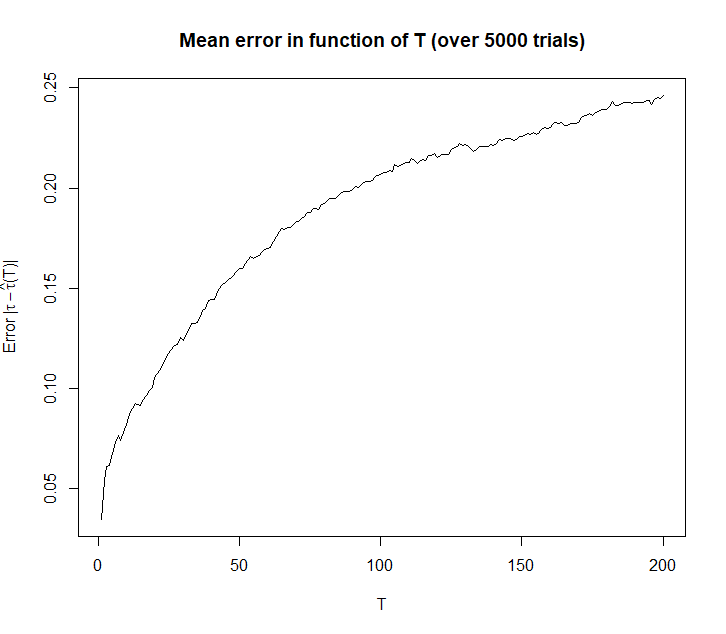}
	\includegraphics[width=0.49\textwidth]{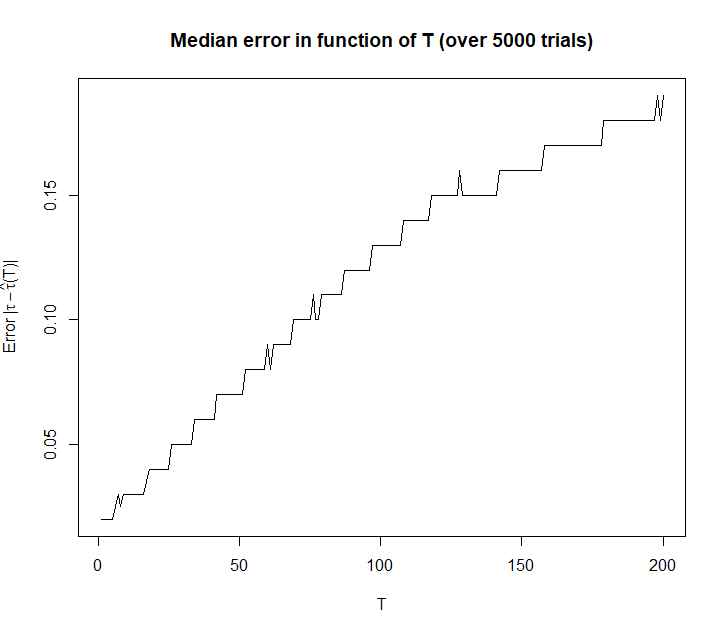}

	\caption{Mean and median of the error over 5000 trials for Model $A$.}\label{fig:fctTA}
\end{figure}

\begin{figure}[H]
	\includegraphics[width=0.49\textwidth]{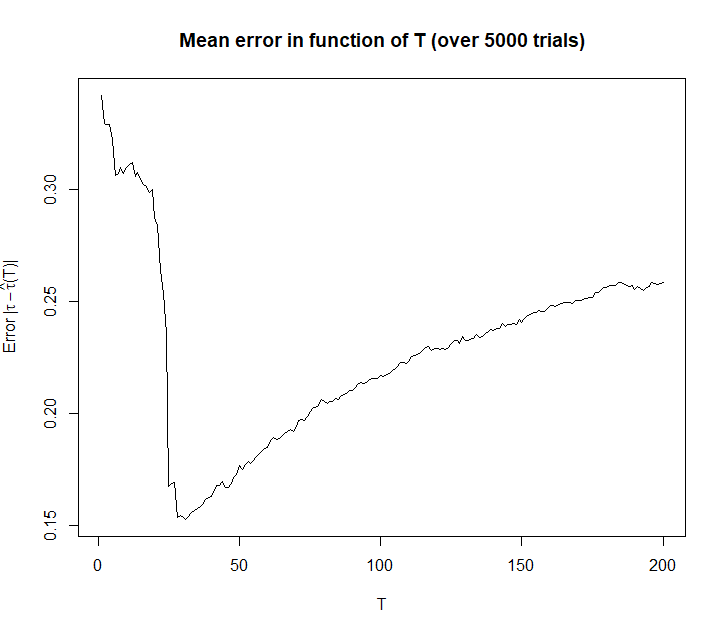}
	\includegraphics[width=0.49\textwidth]{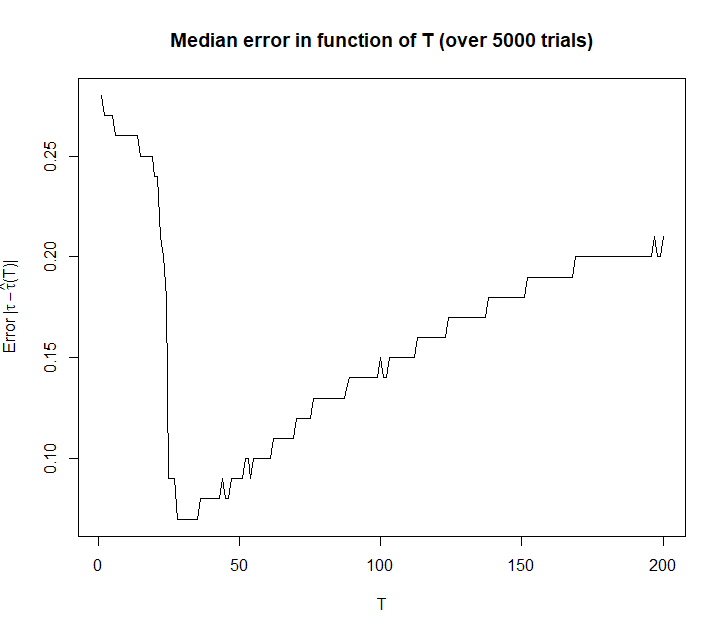}

	\caption{Mean and median of the error over 5000 trials for Model $B$.}\label{fig:fctTB}
\end{figure}

Theorem \ref{hatT} provides a theoretical way to   select  $T$. However, since the statement depends on an unknown tuning constant  $C_\mathcal L$, the theorem cannot be used directly  for choosing $T$ in practice.
In the sequel, two selection procedures for $T$ are investigated, yielding two estimators $\hat T_1$ and $\hat T_2$.
\begin{itemize}
	\item Method 1.  This method is often used to replace the search of tuning constants in adaptive methods. The idea is instead to find a division of the set $\{1,\ldots,d\}$ into $\{1,\ldots,\hat T_1\}$ and its complementary, where  the two subsets are corresponding to 2 ``regimes'' for the data, one of ``big coefficients'', one of small ones.
	\\
	Let $\bar Z^{(T)}=\frac1 T\sum_{j=1}^T Z_j$ and $\bar Z^{(-T)}=\frac1{d-T}\sum_{j=T+1}^d Z_j,$ and consider  $$V(T)=\sum_{j=1}^T (Z_j-\bar Z^{(T)})^2 +\sum_{j=T+1}^d (Z_j-\bar Z^{(-T)})^2.$$
	This quantity $V$ is computed for every $T=1,\dots,d$ and the value $\hat T_1$ is chosen such that $$\hat T_1\in\arg\min_{T=1,\dots,d} V(T).$$ Indeed, this $k$-means-like procedure, by searching for a change-point along $Z_1,\dots, Z_d$, should separate the first most significative  differences $\theta^-_j-\theta^+_j$, $j=1,\dots,\hat T_1$, from the remaining ones, expected to be less significative for estimating $\hat \tau$, in such a way  that keeping for the estimation all components until $\hat T_1$ seems  a reasonable choice.

	\item Method 2. The second idea is more computationally involved and based on subsampling. When performing subsampling, the indices drawn at random are sorted, so that the parameter of interest $\tau$ remains indeed approximatively unchanged. For each $T=1,\dots,d$, we compute  $\hat{\tau}(T)$ for a collection of subsamples. Then, $\hat T_2$ is set to the value of $T$ minimizing the variance of $\hat{\tau}$ over all subsamples. Here, 100 subsamples are built, each of them containing $80\%$ of the initial sample.

\begin{rem}
	
Proportions of data from $50\%$ to $90\%$ have also been tried, with quite similar results. Observe that picking a quite  small proportion of data for subsampling could be interesting  since it provides more variability between the subsamples, but, at the same time, the fact that the ratio between the dimension $d$ and the sample size is modified may be annoying when the aim is to select $T$. We also considered a version of subsampling where a different subsampling index is drawn for every $T=1,\dots,d$: again, this provides more variability in the subsamples, but $\tau$ may also vary more than in the classical version. The results were not significantly different.
\end{rem}

\end{itemize}
The performance of the two methods is compared with the result obtained using the value of $T$ minimizing the average value of $|\tau-\hat\tau(T)|$ over a large number of trials, called hereafter oracle $T^\star$ (here, $T^\star=30$ as obtained above for 5000 trials). Of course, $T^\star$ is not available in practice, since it depends on the true $\tau$. However, it  is introduced as a benchmark. The results, corresponding to 1000 trials, are shown in Figure \ref{fig:comp} and Table \ref{tab:comp}. Observe that the performances of the two methods are very similar, with a slight advantage of Method 2 over Method 1. However, Method 2 is based on subsampling, and, as such, is more CPU-time consuming.

\begin{figure}[H]\centering
	\includegraphics[width=0.7\textwidth]{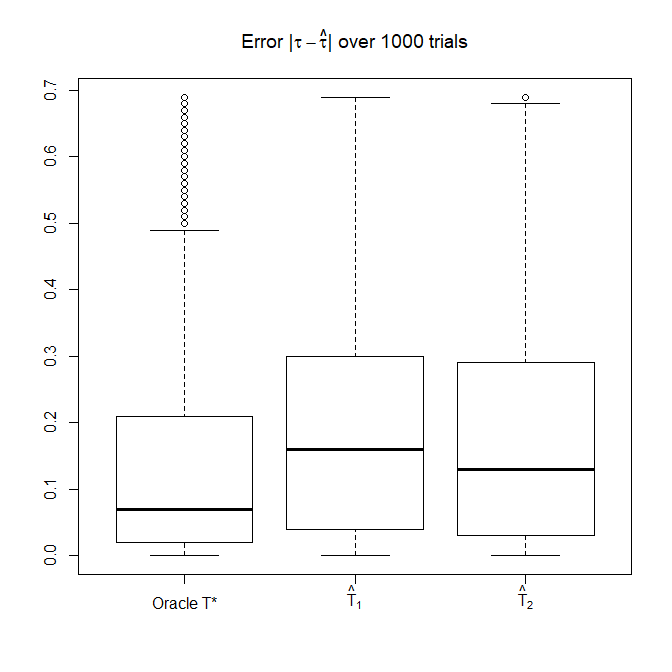}
	\caption{Error of the two selection procedures over 1000 trials, compared with the error obtained using the oracle $T^\star=30$.}\label{fig:comp}
\end{figure}

\begin{table}[H]\centering
	\begin{tabular}{|l|l|l|l|}\hline
Error over 1000 trials	&Oracle $T^\star$ &$\hat{T}_1 $&$ \hat{T}_2$\\\hline
Mean&0.1524&0.2207& 0.2047\\
 (Standard deviation)&(0.18735)&(0.21329)&(0.20841)\\\hline
\end{tabular}
\caption{Mean and standard deviation  over 1000 trials of the error obtained with the oracle $T^\star$ and the two selection methods.}\label{tab:comp}
\end{table}

\section{Proofs}
\subsection{Proof of Proposition \ref{T,s}}
Our proof will heavily rely on standard concentration inequalities, detailed in Appendix (see Section \ref{conc-ineq}).

In the sequel, for the sake of simplicity, we will assume that, additionally to $\eps<\tau<1-\eps$,  $n\tau\in\{2,\ldots,n-2\}$. This will not have any consequence on the result but will avoid unnecessary integer parts.

Also, in this proof, $\Psi_n(T,\Delta)$ will be replaced by $\Psi_n$, when there is no possible confusion.

Let us denote by $P_{(\tau, \theta^+,\theta^-)} $ the probability distribution associated with model \ref{eq:model}.
We will consider the behavior of our estimators under the probability
$P_{(\tau, \theta^+,\theta^-)} $.
 Using the notation $  x^{+}=(x^{+}_1,\ldots,x^{+}_T)$, and $x^{-}=(x^{-}_1,\ldots,x^{-}_T)$, observe that $\hat\tau$ may be defined in the following way:
\begin{multline*}
\hat \tau(T)= \frac 1 n
\argmin_{k\in \{2,\ldots ,n-2\} }
\left\{
\sum_{i=1 }^{k}\sum_{j=1}^T\Big(Y_{i,j}-\frac 1{k}\sum_{i=1}^ {k}Y_{i,j}\Big)^2\right.\\\left.+
\sum_{i=k+1}^n\sum_{j=1 }^T\Big(Y_{i,j}-\frac 1{n-k}\sum_{i= k+1}^n Y_{i,j}\Big)^2\right\}\\
=\argmin_{t\in \{\frac 2n,\ldots ,\frac{n-2}{n} \}}K^T(t).
\end{multline*}

where
 $$K^T(t) =
 \min_{x^-,x^+} L(t,x^-,x^+)-L(\tau,0,0).$$
Here, the function $L$ is given (for $t\in \{\frac 2n,\ldots ,\frac{n-2}{n} \}$) by
$$
L(t,x^-, x^+) =
\sum_{i=1}^ {nt}\sum_{j=1}^T(Y_{i,j}-\theta^{-}_j-x^-_j)^2
+
\sum_{i= nt+1}^n\sum_{j=1}^{T}(Y_{i,j}-\theta^{+}_j-x^+_j)^2.
$$
 Note that $$\frac{dP_{(t,\theta_{+}+x^+,\theta_{-}+x^-)}}{dP_{(\tau,\theta_{+},\theta_{-})}}=\exp\left(-\frac1{2\sigma^2}(L(t,x^-, x^+) -L(\tau,0,0))\right).$$
 Let us consider the case  $t\ge \tau$. The other case can be treated in a symmetrical way.
 
 For $t\ge \tau$, and under  the distribution  $P_{(\tau, \theta^+,\theta^-)} $, we may write
\begin{align*} 
L(t,x^-, x^+) &-L(\tau,0, 0)\\
&=
\sum_{i=1}^ {n\tau}\sum_{j=1}^ T((x^-_j)^2-2\eta_{i,j}x^-_j)+\sum_{i=nt+1}^{n}\sum_{j=1}^T((x^+_j)^2-2\eta_{i,j}x^+_j)\\
&\quad+\sum_{i=n\tau+1}^{nt}\sum_{j=1}^T\left((\theta^+_j-\theta^-_j-x^-_j)^2+2\eta_{i,j}(\theta^+_j-\theta^-_j-x^-_j)\right)
\\ &=
\sum_{i=1}^ {n\tau}\sum_{j=1}^ T((x^-_j)^2-2\eta_{i,j}x^-_j)+\sum_{i=nt+1}^{n}\sum_{j=1}^T((x^+_j)^2-2\eta_{i,j}x^+_j) \\
&\quad+\sum_{i=n\tau+1}^{nt}\sum_{j=1}^T\left((\delta_j-x^-_j)^2+2\eta_{i,j}(\delta_j-x^-_j)\right),\end{align*}
Hence, 
\begin{align*}
L(t,x^-, x^+) &-L(\tau,0, 0)\\&=\sum_{i=nt+1}^{n}\sum_{j=1}^T((x^+_j)^2-2\eta_{i,j}x^+_j)
+\sum_{i=1}^ {nt}\sum_{j=1}^ T((x^-_j)^2-2\eta_{i,j}x^-_j)\\&\quad+\sum_{i=n\tau+1}^{nt}\sum_{j=1}^T\left(\delta_j^2-2\delta_jx^-_j+2\delta_j\eta_{i,j}\right),
 \end{align*}
 where  $\delta=(\delta_1,\dots,\delta_T)$ is the vector $\theta^{+}-\theta^{-}$.
  Now, we have to minimize in $(x^-, x^+)$ this expression
 \begin{multline*}
 \sum_{i=nt+1}^{n}\sum_{j=1}^T((x^+_j)^2-2\eta_{i,j}x^+_j)
 +\sum_{i=1}^ {nt}\sum_{j=1}^ T((x^-_j)^2-2\eta_{i,j}x^-_j)\\+\sum_{i=n\tau+1}^{nt}\sum_{j=1}^T\left(\delta_j^2-2\delta_jx^-_j+2\delta_j\eta_{i,j}\right).
 \end{multline*}

 The minimum  is attained  by taking, for every $j$, 
 \begin{align*}
 \hat x^+_j&=\frac{\sum_{i=nt+1}^n\eta_{i,j}}{n-nt},\\
 \hat x^-_j&=\frac{\sum_{i=1}^{nt}\eta_{i,j} +(nt-n\tau)\delta_j}{nt}.
 \end{align*}
  So, the minimum is
  \begin{multline*}
 K^T(t)=\sum_{j=1}^T\left(-\frac{\left(\sum_{i=nt+1}^n\eta_{i,j}\right)^2 }{n-nt}-\frac{\left(\sum_{i=1}^ {nt}\eta_{i, j}+(nt-n\tau)\delta_j\right)^2}{nt}\right.\\\left.+(nt-n\tau)\delta_j^2+2\delta_j\sum_{i=n\tau+1}^{nt}\eta_{i,j}\right).
 \end{multline*}

Under $P_{(\tau, \theta^+,\theta^-)} $, $K^T(t)$ can be written in the following way:
$$
K^T(t) = -\sum_{j=1}^T \sigma^2V_j^2(t)-\sum_{j=1}^T \sigma^2W_j^{2}(t)+\sum_{j=1}^T \delta_j^2\frac{(nt-n\tau )n\tau}{nt}+2N_1(t)-2N_2(t),
$$
where 
\begin{align*}
\sigma^2V_j^2(t)&=\frac{\left(\sum_{i=nt+1}^n\eta_{i,j}\right)^2}{n-nt},\\
\sigma^2W_j^{2}(t)&=\frac{\left(\sum_{i=1}^ {nt}\eta_{i,j}\right)^2}{nt},\\
N_1(t)&=\sum_{j=1}^T\sum_{i=n\tau+1}^{ nt}\eta_{i,j}\delta_ j,\\
N_2(t)&=\sum_{j=1}^T\frac{\sum_{i=1}^{ nt}\eta_{i,j}(nt-n\tau)\delta_ j}{nt}\\
N_1(\tau)&=N_2(\tau)=0.
 \end{align*}
Observe that $V_j^2(t)$, $j=1,\dots,T$, are independent $\chi^2(1)$ random variables,  as well as 
$W_j^{2}(t)$, $j=1,\dots,T$. Moreover,   $N_1(t)\sim \Nr\left(0,\sum_{j=1}^T\sigma^2(nt-n\tau)\delta_j^2\right)$, 
$N_2(t)\sim \Nr\left(0,\sum_{j=1}^T\frac{\sigma^2(nt-n\tau)^2\delta_j^2}{nt}\right)$.

We have
\begin{align*} 
	&P\left(|\hat \tau-\tau|\ge \lambda \Psi_n\right)
	\\
	&= P\left(\inf_{|\frac kn-\tau|\ge \lambda \Psi_n}K^T\left(\frac kn\right)<
	\inf_{|\frac kn-\tau|< \lambda \Psi_n}K^T\left(\frac kn\right)\right)
	\\
	&\leq P\left(\inf_{|\frac kn-\tau|\ge \lambda \Psi_n}K^T\left(\frac kn\right)<
	K^T\left(\tau\right)\right)
	\\
	&\leq P\left(\inf_{\frac kn-\tau\ge \lambda \Psi_n}K^T\left(\frac kn\right)<
	K^T\left(\tau\right)\right)+
	P\left(\inf_{\frac kn-\tau\le -\lambda \Psi_n}K^T\left(\frac kn\right)<
	K^T\left(\tau\right)\right).
\end{align*}
We will only consider the first term, the other one can be treated in a symmetrical way.
We have

\begin{align*}
	&P\left(\inf_{\frac kn-\tau\ge \lambda \Psi_n}K^T\left(\frac kn\right)<
	K^T\left(\tau\right)\right)
	\\
	&= P\left(\exists k\in \{2,\ldots, n-2\}, \; \frac kn-\tau\ge \lambda \Psi_n,\; K^T\left(\frac kn\right)<K^T\left(\tau\right)\right)
	\\
	&\le P\Bigg(\exists k\in \{n\tau+n\lambda\Psi_n,\ldots, n-2\},\Bigg. \\ &\quad-\sum_{j=1}^T V_j^2\left(\frac kn\right)-\sum_{j=1}^T W_j^{2}\left(\frac kn\right)
	+\frac{
		2N_1(\frac kn)-2N_2(\frac kn)}{\sigma^2}+\frac{n\Delta_T^2}{\sigma^2}
	\frac{(\frac kn -\tau)n\tau}{k}	\\
	&\Bigg.
	<-\sum_{j=1}^T V_j^2\left(\tau\right)-\sum_{j=1}^T W_j^{2}\left(\tau\right)+
	\frac{
		2N_1(\tau)-2N_2(\tau)}{\sigma^2}
	\Bigg)
	\\
	&\le P\left(\exists k\in \{n\tau+n\lambda\Psi_n,\ldots, n-2\}, \sum_{j=1}^T V_j^2\left(\frac kn\right)+\sum_{j=1}^T W_j^{2}\left(\frac kn\right)
	\right.
	\\
	&\left.\quad-\sum_{j=1}^T V_j^2\left(\tau\right)-\sum_{j=1}^T W_j^{2}\left(\tau\right)-\frac{
		2N_1(\frac kn)-2N_2(\frac kn)}{\sigma^2}
		>\frac{n\Delta_T^2}{\sigma^2}
	\frac{(\frac kn -\tau)n\tau}{k}
	\right),
\end{align*}

since $N_1(\tau)=N_2(\tau)=0$.
Thus, 
\begin{align*}
&P\left(\inf_{\frac kn-\tau\ge \lambda \Psi_n}K^T\left(\frac kn\right)<
K^T\left(\tau\right)\right)
\\
&\le P\Bigg(\exists k\in \{n\tau+n\lambda\Psi_n,\ldots, n-2\}, \Bigg.\nonumber
\\
&\Bigg.\quad \sum_{j=1}^T V_j^2\left(\frac kn\right)+\sum_{j=1}^T W_j^{2}\left(\frac kn\right)
-\sum_{j=1}^T V_j^2\left(\tau\right)-\sum_{j=1}^T W_j^{2}\left(\tau\right)
>\frac{n\Delta_T^2}{\sigma^2}
\frac{(\frac kn -\tau)n\tau}{2k}	\Bigg)\nonumber
\\
&\quad+ P\left(\exists k\in \{n\tau+n\lambda\Psi_n,\ldots, n-2\}, 
\frac{|
	2N_1(\frac kn)-2N_2(\frac kn)|}{\sigma^2}>\frac{n\Delta_T^2}{\sigma^2}
\frac{(\frac kn -\tau)n\tau}{2k}\label{eq:prop1sum}
\right)\nonumber\\&:=P_1+P_2.
\end{align*}

Furthermore, for the second term $P_2$, using \eref{Gaussconc}
\begin{align*}
&P\left(\exists  k\in \{n\tau+n\lambda\Psi_n,\ldots, n-2\},\; \frac{|N_1(\frac kn)-N_2(\frac kn)|}{\sigma^2}>\frac{\Delta_T^2}{\sigma^2}\frac{n(\frac kn-\tau)n\tau}{4 k}\right)
\\
&
\qquad\le \sum_{k\in \{n\tau+n\lambda\Psi_n,\ldots, n-2\}} P\left( \left|N_1\Big(\frac kn\Big)\right|>\frac{\Delta_T^2n(\frac kn-\tau)n\tau}{8 k}\right)\\&\qquad+
\sum_{k\in \{n\tau+n\lambda\Psi_n,\ldots, n-2\}}P\left( \left|N_2\Big(\frac kn\Big)\right|>\frac{\Delta_T^2n(\frac kn-\tau)n\tau}{8 k}\right)
\\
&\quad\le \sum_{k\in \{n\tau+n\lambda\Psi_n,\ldots, n-2\}}
2\exp\left(
-\frac{\left(\frac{\Delta_T^2n(\frac kn-\tau)n\tau}{8 k}\right)^2}{2n\Delta_T^2(\frac kn-\tau)\sigma^2}
\right)
+ 2\exp\left(-\frac{\left(\frac{\Delta_T^2n(\frac kn-\tau)n\tau}{8 k}\right)^2}{2\frac{n\Delta_T^2(\frac kn-\tau)^2\sigma^2}{\frac kn}}\right)
\\
&
\quad\le 2n\left[ \exp\left( - \frac{ \tau^2 n\Delta_T^2\Psi_n\lambda}{64\sigma^2}\right) +\exp\left(- \frac{\tau^2 n\Delta_T^2}{64\sigma^2}\right)\right]
\\
&
\quad\le 2n\left[ \exp\left( - \frac{ \tau^2 \lambda}
{64}\right) \wedge\exp\left( - \frac{ \tau^2\lambda T\sigma^2}{64n\Delta_T^2}\right)+\exp\left(- \frac{\tau^ 2n\Delta_T^2}{64\sigma^2}\right)\right]
\\
&
\quad\le 2n\left[\exp\left( - \frac{ \tau^2 \lambda}
{64}\right) +\exp\left(- \frac{\tau^ 2n\Delta_T^2}{64\sigma^2}\right)\right].
\end{align*}

To control the first term $P_1$, we  distinguish two situations.
We begin with investigating the case where 
$n\Delta_T^2\le 32T\sigma^2/\eps^2.$

Then, using Lemma \ref{chi-deux-MD},
\begin{align*}
&P\Bigg(\exists k\in \{n\tau+n\lambda\Psi_n,\ldots, n-2\}, \Bigg.
\\
&\quad\Bigg. \sum_{j=1}^T V_j^2\left(\frac kn\right)+\sum_{j=1}^T W_j^{2}\left(\frac kn\right)
-\sum_{j=1}^T V_j^2\left(\tau\right)-\sum_{j=1}^T W_j^{2}\left(\tau\right)>\frac{n\Delta_T^2}{\sigma^2}
\frac{(\frac kn -\tau)n\tau}{2k}	\Bigg)
\\
&\le P\Bigg(\exists k\in \{n\tau+n\lambda\Psi_n,\ldots, n-2\}, \Bigg.
\\
&\Bigg. \quad\sum_{j=1}^T \left(V_j^2\left(\frac kn\right)-1\right)+\sum_{j=1}^T\left( W_j^{2}\left(\frac kn\right)-1\right)
>\frac{n\Delta_T^2}{\sigma^2}
\frac{(\frac kn -\tau)n\tau}{4k}
\Bigg)\\&\qquad+P\left(\sum_{j=1}^T \left(V_j^2(\tau)-1\right)+\sum_{j=1}^T\left( W_j^{2}(\tau)-1\right)>\frac{n\Delta_T^2}{\sigma^2}
\frac{\lambda\Psi_nn\tau}{4n}\right).
\end{align*}
Considering the first term, we have
\begin{align*}&P\Bigg(\exists k\in \{n\tau+n\lambda\Psi_n,\ldots, n-2\}, \Bigg.
\\
&\Bigg. \qquad\sum_{j=1}^T V_j^2\left(\frac kn\right)+\sum_{j=1}^T W_j^{2}\left(\frac kn\right)
-\sum_{j=1}^T V_j^2\left(\tau\right)-\sum_{j=1}^T W_j^{2}\left(\tau\right)>\frac{n\Delta_T^2}{\sigma^2}
\frac{(\frac kn -\tau)n\tau}{2k}	\Bigg)\\
&\le 2\sum_{k\in \{n\tau+n\lambda\Psi_n,\ldots, n-2\}}P\left( \sum_{j=1}^T \left(V_j^2\left(\frac kn\right)-1\right)
>\frac{n\Delta_T^2}{\sigma^2}
\frac{(\frac kn -\tau)n\tau}{8k}
\right)
\\
&\le 2\sum_{k\in \{n\tau+n\lambda\Psi_n,\ldots, n-2\}}\exp\left(-\left(\frac{n\Delta_T^2}{\sigma^2}
\frac{(\frac kn -\tau)n\tau}{8k}\right)^2\frac1{16T}\right)
\\
&\le 2n	\exp\left(-\frac{\tau^2\lambda^2 \sigma^4T}{1024(n\Delta_T^2)^2}\right).
\end{align*}
As well, 
\begin{align*}
&P\left(\sum_{j=1}^T (V_j^2\left(\tau\right)-1)+\sum_{j=1}^T (W_j^{2}\left(\tau\right)-1)>\frac{n\Delta_T^2}{\sigma^2}
\frac{\lambda\Psi_nn\tau}{4n}
\right)\le 2	\exp\left(-\frac{\tau^2\lambda^2\sigma^4 T}{1024(n\Delta_T^2)^2}\right).
\end{align*}

The two preceding bounds lead to a bound $\exp(-\frac{\tau^2\lambda}{1024})$, in the case where $\frac{\sigma^4T}{(n\Delta_T^2)^2}\ge \frac1\lambda$.

Now let us investigate the more intricate case where $n\Delta_T^2\le 32T\sigma^2/\eps^2$ (still) but  $\frac{\sigma^4T}{(n\Delta_T^2)^2}\le \frac1\lambda$ (i.e. $\frac{(n\Delta_T^2)^2}{\sigma^4T}\ge \lambda$).

We have

\begin{align*}
&P\Bigg(\exists k\in \{n\tau+n\lambda\Psi_n,\ldots, n-2\}, \Bigg.
\\
&\Bigg.\qquad \sum_{j=1}^T V_j^2\left(\frac kn\Bigg)+\sum_{j=1}^T W_j^{2}\left(\frac kn\right)
-\sum_{j=1}^T V_j^2\left(\tau\right)-\sum_{j=1}^T W_j^{2}\left(\tau\right)>\frac{n\Delta_T^2}{\sigma^2}
\frac{(\frac kn -\tau)n\tau}{2k}	\right)
\\
&
\le  P\left(\exists k\in \{n\tau+n\lambda\Psi_n,\ldots, n-2\},  \sum_{j=1}^T V_j^2\left(\frac kn\right)-\sum_{j=1}^T V_j^2\left(\tau\right)>\frac{n\Delta_T^2}{\sigma^2}
\frac{(\frac kn -\tau)n\tau}{4k}	\right)
\\
&\quad+P\left(\exists k\in \{n\tau+n\lambda\Psi_n,\ldots, n-2\},  \sum_{j=1}^T W_j^2\left(\frac kn\right)-\sum_{j=1}^T W_j^2\left(\tau\right)>\frac{n\Delta_T^2}{\sigma^2}
\frac{(\frac kn -\tau)n\tau}{4k}	\right).
\end{align*} 
Let us compute
\begin{align*}
\sigma^2\left(W_j^2\left(\frac kn\right)- W_j^2\left(\tau\right)\right)&= \left(\sum_{i=1}^{n\tau}\eta_{i,j}\right)^2\left(\frac1k-\frac1{n\tau}\right)+\left(\sum_{i=n\tau+1}^{k}\eta_{i,j}\right)^2\frac1k
\\&\quad+\frac2k\left(\sum_{i=1}^{n\tau}\eta_{i,j}\right)\left(\sum_{i=n\tau+1}^{k}\eta_{i,j}\right)
\\
&\leq\frac{k-n\tau}k\left[\left(\sum_{i=n\tau+1}^{k}\eta_{i,j}\right)^2\frac1{k-n\tau}-\left(\sum_{i=1}^{n\tau}\eta_{i,j}\right)^2\frac1{n\tau}\right]
\\&\quad+\left|\frac2k\sum_{i=1}^{n\tau}\eta_{i,j}\sum_{i=n\tau+1}^{k}\eta_{i,j}\right|.
\end{align*}
As well,
\begin{align*}
\sigma^2\left(V_j^2\left(\frac kn\right)- V_j^2\left(\tau\right)\right)&= \left(\sum_{i=k+1}^{n}\eta_{i,j}\right)^2\left(\frac1{n-k}-\frac1{n-n\tau}\right)-\left(\sum_{i=n\tau+1}^{k}\eta_{i,j}\right)^2\frac1{n-n\tau}
\\&\quad-\frac2{n-n\tau}\left(\sum_{i=k+1}^{n}\eta_{i,j}\right)\left(\sum_{i=n\tau+1}^{k}\eta_{i,j}\right)
\\
&\leq
\frac{k-n\tau}{n-n\tau}\left[\left(\sum_{i=k+1}^{n}\eta_{i,j}\right)^2\frac{1}{n-k}-\left(\sum_{i=n\tau+1}^{k}\eta_{i,j}\right)^2\frac1{k-n\tau}\right]
\\&\quad+
\left|\frac2{n-n\tau}\sum_{i=k+1}^{n}\eta_{i,j}\sum_{i=n\tau+1}^{k}\eta_{i,j}\right|
\end{align*}

As a consequence, we get
\begin{align*}&P\left(\exists k\in \{n\tau+n\lambda\Psi_n,\ldots, n-2\}, \sum_{j=1}^T W_j^2\left(\frac kn\right)-\sum_{j=1}^T W_j^2\left(\tau\right)>\frac{n\Delta_T^2}{\sigma^2}
\frac{(\frac kn -\tau)n\tau}{4k}	\right)
\\
&\quad\le \sum_{k\in \{n\tau+n\lambda\Psi_n,\ldots, n-2\}}P\Bigg(\sum_{j=1}^T\frac{k-n\tau}k\left[\left(\sum_{i=n\tau+1}^{k}\eta_{i,j}\right)^2\frac1{k-n\tau}-\left(\sum_{i=1}^{n\tau}\eta_{i,j}\right)^2\frac1{n\tau}\right]\Bigg.\\\Bigg.&\kern10.5cm>n\Delta_T^2
\frac{(\frac kn -\tau)n\tau}{8k}	\Bigg)
\\
&\qquad+\sum_{k\in \{n\tau+\lambda\Psi_n,\ldots, n-2\}} P\left(
\bigg|\sum_{j=1}^T\frac2k\sum_{i=1}^{n\tau}\eta_{i,j}\sum_{i=n\tau+1}^{k}\eta_{i,j}\bigg|>n\Delta_T^2
\frac{(\frac kn -\tau)n\tau}{8k}	\right).
\end{align*}
As well,
\begin{align*}
&P\left(\exists k\in \{n\tau+n\lambda\Psi_n,\ldots, n-2\},  \sum_{j=1}^T V_j^2\left(\frac kn\right)-\sum_{j=1}^T V_j^2\left(\tau\right)>\frac{n\Delta_T^2}{\sigma^2}
\frac{(\frac kn -\tau)n\tau}{4k}	\right)
\\
&\le \sum_{k\in \{n\tau+\lambda\Psi_n,\ldots, n-2\}}P\Bigg(\sum_{j=1}^T
\frac{k-n\tau}{n-n\tau}\left[\left(\sum_{i=k+1}^{n}\eta_{i,j}\right)^2\frac{1}{n-k}-\left(\sum_{i=n\tau+1}^{k}\eta_{i,j}\right)^2\frac1{k-n\tau}\right]
\Bigg.\\\Bigg.&\kern10.5cm>n\Delta_T^2
\frac{(\frac kn -\tau)n\tau}{8k}	\Bigg)
\\
&\quad+\sum_{k\in \{n\tau+\lambda\Psi_n,\ldots, n-2\}} P\left(
\bigg|\sum_{j=1}^T\frac2{n-n\tau}\sum_{i=k+1}^{n}\eta_{i,j}\sum_{i=n\tau+1}^{k}\eta_{i,j}\bigg|
>n\Delta_T^2
\frac{(\frac kn -\tau)n\tau}{8k}	\right).
\end{align*}

Now,  
\begin{align*}
&\sum_{k\in \{n\tau+n\lambda\Psi_n,\ldots, n-2\}}P\Bigg(\sum_{j=1}^T\frac{k-n\tau}k\left[\left(\sum_{i=n\tau+1}^{k}\eta_{i,j}\right)^2\frac1{k-n\tau}-\left(\sum_{i=1}^{n\tau}\eta_{i,j}\right)^2\frac1{n\tau}\right]\Bigg.\\\Bigg.&\kern10.5cm>n\Delta_T^2
\frac{(\frac kn -\tau)n\tau}{8k}	\Bigg)
\\&=\sum_{k\in \{n\tau+n\lambda\Psi_n,\ldots, n-2\}}P\Bigg(\sum_{j=1}^T\frac{k-n\tau}k\left[\left(\sum_{i=n\tau+1}^{k}\eta_{i,j}\right)^2\frac1{\sigma^2(k-n\tau)}-\left(\sum_{i=1}^{n\tau}\eta_{i,j}\right)^2\frac1{\sigma^2n\tau}\right]\Bigg.\\\Bigg.&\kern10.5cm>\frac{n\Delta_T^2}{\sigma^2}
\frac{(\frac kn -\tau)n\tau}{8k}	\Bigg)
\\
&
\le
\sum_{k\in \{n\tau+n\lambda\Psi_n,\ldots, n-2\}}P\left(\sum_{j=1}^T\bigg|\left(\sum_{i=n\tau+1}^{k}\eta_{i,j}\right)^2\frac1{\sigma^2(k-n\tau)}-1\bigg|>\frac{n\Delta_T^2}{\sigma^2}
\frac{\tau}{16}	\right)
\\
&\quad+
\sum_{k\in \{n\tau+n\lambda\Psi_n,\ldots, n-2\}}P\left(\sum_{j=1}^T\bigg|\left(\sum_{i=1}^{n\tau}\eta_{i,j}\right)^2\frac1{\sigma^2n\tau}-1\bigg|>\frac{n\Delta_T^2}{\sigma^2}
\frac{\tau}{16}	\right)
\\
&\le 2n\exp\left(- \frac{(n\Delta_T^2\tau)^2}{\sigma^4 256T}\right)
\\
&
\le 2n\exp\left(- \frac{\tau^2\lambda}{256}\right).
\end{align*}
In the last two bounds, we have applied Lemma \ref{chi-deux-MD}, then used the fact that we are in the case  $\frac{(n\Delta_T^2)^2}{\sigma^4T}\ge \lambda$. 
As well, 
\begin{align*}
&\sum_{k\in \{n\tau+\lambda\Psi_n,\ldots, n-2\}}P\Bigg(\sum_{j=1}^T
\frac{k-n\tau}{n-n\tau}\left[\left(\sum_{i=k+1}^{n}\eta_{i,j}\right)^2\frac{1}{n-k}-\left(\sum_{i=n\tau+1}^{k}\eta_{i,j}\right)^2\frac1{k-n\tau}\right]\Bigg.\\\Bigg.&\kern10.5cm
>n\Delta_T^2
\frac{(\frac kn -\tau)n\tau}{8k}	\Bigg)
\\
&\le
\sum_{k\in \{n\tau+n\lambda\Psi_n,\ldots, n-2\}}P\left(\sum_{j=1}^T\bigg|\left(\sum_{i=k+1}^{n}\eta_{i,j}\right)^2\frac1{\sigma^2(n-k)}-1\bigg|>\frac{n\Delta_T^2}{\sigma^2}
\frac{(1 -\tau)n\tau}{16k}\right)
\\
&\quad+
\sum_{k\in \{n\tau+n\lambda\Psi_n,\ldots, n-2\}}P\left(\sum_{j=1}^T\bigg|\left(\sum_{i=n\tau+1}^{k}\eta_{i,j}\right)^2\frac1{\sigma^2(k-n\tau)}-1\bigg|>\frac{n\Delta_T^2}{\sigma^2}
\frac{(1 -\tau)n\tau}{16k}\right)\\&\le  2n \exp \left(-\frac{\tau^2(1-\tau)^2(n\Delta_T^2)^2}{256\sigma^4T}\right)
\\
&\le 2n\exp\left(- \frac{\tau^2(1-\tau)^2\lambda}{256}\right).	
\end{align*}

Now, let us denote by $\F$, the $\sigma-$algebra spanned by the variables $\{\eta_{i,j},\; i\le n\tau,\; j\le T\}$. We write
\begin{align*}
&\sum_{k\in \{n\tau+n\lambda\Psi_n,\ldots, n-2\}} P\left(
\bigg|\sum_{j=1}^T\frac2k\sum_{i=1}^{n\tau}\eta_{i,j}\sum_{i=n\tau+1}^{k}\eta_{i,j}\bigg|>n\Delta_T^2
\frac{(\frac kn -\tau)n\tau}{8k}	\right)
\\
&\quad=	\sum_{k\in \{n\tau+n\lambda\Psi_n,\ldots, n-2\}}E\left[P\left(\bigg|\sum_{j=1}^T\sum_{i=1}^{n\tau}\eta_{i,j}\sum_{i=n\tau+1}^{k}\eta_{i,j}\bigg|>n\Delta_T^2
\frac{(\frac kn -\tau)n\tau}{16}\Big|\F	\right)\right].
\end{align*}
Conditionally on $\F$, the random variable $\sum_{j=1}^T\sum_{i=1}^{n\tau}\eta_{i,j}\sum_{i=n\tau+1}^{k}\eta_{i,j}$ follows a centered normal distribution $\Nr(0,\sigma^2(k-n\tau)\sum_{j=1}^{T}(\sum_{i=1}^{n\tau}\eta_{i,j})^2)$, that is $$\frac{\sum_{j=1}^T\sum_{i=1}^{n\tau}\eta_{i,j}\sum_{i=n\tau+1}^{k}\eta_{i,j}}{\sigma(k-n\tau)^{1/2}(\sum_{j=1}^{T}(\sum_{i=1}^{n\tau}\eta_{i,j})^2)^{1/2}}\sim\Nr(0,1).$$
Thus,
\begin{align*}
&\sum_{k\in \{n\tau+n\lambda\Psi_n,\ldots, n-2\}} P\left(
\bigg|\sum_{j=1}^T\frac2k\sum_{i=1}^{n\tau}\eta_{i,j}\sum_{i=n\tau+1}^{k}\eta_{i,j}\bigg|>n\Delta_T^2
\frac{(\frac kn -\tau)n\tau}{8k}\right)\\
&\le 	2\sum_{k\in \{n\tau+n\lambda\Psi_n,\ldots, n-2\}}E\left[\exp\left(-\frac{(n\Delta_T^2(\frac kn -\tau)n\tau)^2}{16^2}\frac1{2\sigma^2(k-n\tau)\sum_{j=1}^T(\sum_{i=1}^{n\tau}\eta_{i,j})^2}\right)	\right]\\
&\le 2	\sum_{k\in \{n\tau+n\lambda\Psi_n,\ldots, n-2\}}E\Bigg[\exp\left(-\frac{(n\Delta_T^2(\frac kn -\tau)n\tau)^2}{16^2}\frac1{2\sigma^2(k-n\tau)\sum_{j=1}^T(\sum_{i=1}^{n\tau}\eta_{i,j})^2}\right)\Bigg.\\\Bigg.&\kern9.5cm\times	\I_{\left\{\sum_{j=1}^T\frac{\left(\sum_{i=1}^{n\tau}\eta_{i,j}\right)^2}{\sigma^2n\tau}\le 8T\right\}}\Bigg]
\\
&+	2nP\left(\sum_{j=1}^T\frac{\left(\sum_{i=1}^{n\tau}\eta_{i,j}\right)^2}{\sigma^2n\tau}\ge 8T\right)
\\
&\le 2n\exp\left(-\frac{\lambda\tau}{16^3}\right)
+ 2n\exp\left(-\frac{n\Delta_T^2\eps^2}{32\sigma^2}
\right).
\end{align*}
We used here $n\Delta_T^2\le 32T\sigma^2/\epsilon^2$ together with 
lemma \ref{chi2}.
To end the proof of this part, we investigate the last term: let now $\F_k$ denote the $\sigma-$algebra spanned by the variables $\{\eta_{i,j},\; i> k,\; j\le T\}$, and using again lemma \ref{chi2}. We write:

\begin{align*}
&\sum_{k\in \{n\tau+n\lambda\Psi_n,\ldots, n-2\}}  P\left(
\bigg|\sum_{j=1}^T\frac2{n-n\tau}\sum_{i=k+1}^{n}\eta_{i,j}\sum_{i=n\tau+1}^{k}\eta_{i,j}\bigg|>{n\Delta_T^2}
\frac{(\frac kn -\tau)n\tau}{8k}	\right)
\\
&
\le \sum_{k\in \{n\tau+n\lambda\Psi_n,\ldots, n-2\}}E\left[P\left(
\bigg|\sum_{j=1}^T\frac 1{n-n\tau}\sum_{i=k+1}^{n}\eta_{i,j}\sum_{i=n\tau+1}^{k}\eta_{i,j}\bigg|>{n\Delta_T^2}
\frac{(\frac kn -\tau)n\tau}{16k}\Big|\F_k	\right)\right]
\\
&\le 2	\sum_{k\in \{n\tau+n\lambda\Psi_n,\ldots, n-2\}}E\Bigg[\exp\left(-\frac{(n\Delta_T^2(\frac kn -\tau)(n-n\tau)n\tau)^2}{16^2k^2}\frac1{2\sigma^2(k-n\tau)\sum_{j=1}^T(\sum_{i=k+1}^{n}\eta_{i,j})^2}\right)\Bigg.
\\
&	\Bigg.\kern10cm\times\I_{\left\{\sum_{j=1}^T\frac{\left(\sum_{i=k+1}^{n}\eta_{i,j}\right)^2}{\sigma^2(n-k)}\le 8T\right\}}\Bigg]
\\
&+	2nP\left(\sum_{j=1}^T\frac{\left(\sum_{i=k+1}^{n}\eta_{i,j}\right)^2}{\sigma^2(n-k)}\ge 8T\right)
\le 2n\exp\left(-\frac{\lambda\tau^2(1-\tau)^2}{16^3}\right)
+ 2n\exp\left(-\frac{n\Delta_T^2\eps^2}{64\sigma^2}\right).
\end{align*}

We now investigate the case where   $n\Delta_T^2\ge 32T\sigma^2/\eps^2.$
Note that, as $\eps<1/2$, in this case, we also have $n\Delta_T^2\ge 64T\sigma^2/\eps.$
We have:

\begin{align*}
&P\Bigg(\exists k\in \{n\tau+n\lambda\Psi_n,\ldots, n-2\},\Bigg.\\
&\Bigg. \qquad\sum_{j=1}^T V_j^2\left(\frac kn\right)+\sum_{j=1}^T W_j^{2}\left(\frac kn\right)
-\sum_{j=1}^T V_j^2\left(\tau\right)-\sum_{j=1}^T W_j^{2}\left(\tau\right)
>\frac{n\Delta_T^2}{\sigma^2}
\frac{(\frac kn -\tau)n\tau}{2k}	\Bigg)
\\
&
\le P\left(\exists k\in \{n\tau+\lambda\Psi_n,\ldots, n-2\}, \sum_{j=1}^T V_j^2\left(\frac kn\right)-\sum_{j=1}^T V_j^2\left(\tau\right)>\frac{n\Delta_T^2}{\sigma^2}
\frac{(\frac kn -\tau)n\tau}{4k}	\right)
\\
&\quad +P\left(\exists k\in \{n\tau+n\lambda\Psi_n,\ldots, n-2\},  \sum_{j=1}^T W_j^2\left(\frac kn\right)-\sum_{j=1}^T W_j^2\left(\tau\right)>\frac{n\Delta_T^2}{\sigma^2}
\frac{(\frac kn -\tau)n\tau}{4k}	\right)
\end{align*} 
We will use the next upper bounds.
We have:
\begin{align*}
\sigma^2\left(W_j^2\left(\frac kn\right)- W_j^2\left(\tau\right)\right)&= \left(\sum_{i=1}^{n\tau}\eta_{i,j}\right)^2\left(\frac1k-\frac1{n\tau}\right)+\left(\sum_{i=n\tau+1}^{k}\eta_{i,j}\right)^2\frac1k\\&\quad
+\frac2k\left(\sum_{i=1}^{n\tau}\eta_{i,j}\right)\left(\sum_{i=n\tau+1}^{k}\eta_{i,j}\right)
\\
&\leq\left(\sum_{i=n\tau+1}^{k}\eta_{i,j}\right)^2\frac1k+\left|\frac2k\sum_{i=1}^{n\tau}\eta_{i,j}\sum_{i=n\tau+1}^{k}\eta_{i,j}\right|,
\end{align*}
as well as 
\begin{align*}
\sigma^2\left(V_j^2\left(\frac kn\right)- V_j^2\left(\tau\right)\right)&= \left(\sum_{i=k+1}^{n}\eta_{i,j}\right)^2\left(\frac1{n-k}-\frac1{n-n\tau}\right)-\left(\sum_{i=n\tau+1}^{k}\eta_{i,j}\right)^2\frac1{n-n\tau}
\\&\quad-\frac2{n-n\tau}\left(\sum_{i=k+1}^{n}\eta_{i,j}\right)\left(\sum_{i=n\tau+1}^{k}\eta_{i,j}\right)
\\
&\leq
\quad \left(\sum_{i=k+1}^{n}\eta_{i,j}\right)^2\frac{k-n\tau}{(n-k)(n-n\tau)}\\&\quad+
\left|\frac2{n-n\tau}\sum_{i=k+1}^{n}\eta_{i,j}\sum_{i=n\tau+1}^{k}\eta_{i,j}\right|.
\end{align*}

As a consequence, we get

\begin{align*}
&P\left(\exists k\in \{n\tau+n\lambda\Psi_n,\ldots, n-2\},  \sum_{j=1}^T W_j^2\left(\frac kn\right)-\sum_{j=1}^T W_j^2\left(\tau\right)>\frac{n\Delta_T^2}{\sigma^2}
\frac{(\frac kn -\tau)n\tau}{4k}	\right)
\\
&\quad\le \sum_{k\in \{n\tau+n\lambda\Psi_n,\ldots, n-2\}}P\left(\frac1{\sigma^2}\sum_{j=1}^T\left(\sum_{i=n\tau+1}^{k}\eta_{i,j}\right)^2\frac1k>\frac{n\Delta_T^2}{\sigma^2}
\frac{(\frac kn -\tau)n\tau}{8k}	\right)
\\
&\qquad+P\left(\frac1{\sigma^2}
\sum_{j=1}^T\left|\frac2k\sum_{i=1}^{n\tau}\eta_{i,j}\sum_{i=n\tau+1}^{k}\eta_{i,j}\right|>\frac{n\Delta_T^2}{\sigma^2}
\frac{(\frac kn -\tau)n\tau}{8k}	\right).
\end{align*}
As well,
\begin{align*}
&P\left(\exists k\in \{n\tau+n\lambda\Psi_n,\ldots, n-2\}, \sum_{j=1}^T V_j^2\left(\frac kn\right)-\sum_{j=1}^T V_j^2\left(\tau\right)>\frac{n\Delta_T^2}{\sigma^2}
\frac{(\frac kn -\tau)n\tau}{4k}	\right)
\\
&\le \sum_{k\in \{n\tau+n\lambda\Psi_n,\ldots, n-2\}}P\left(\frac1{\sigma^2}\sum_{j=1}^T\left(\sum_{i=k+1}^{n}\eta_{i,j}\right)^2\frac{k-n\tau}{(n-k)(n-n\tau)}>\frac{n\Delta_T^2}{\sigma^2}
\frac{(\frac kn -\tau)n\tau}{8k}	\right)
\\
&\qquad+P\left(\frac1{\sigma^2}
\sum_{j=1}^T\left|\frac2{n-n\tau}\sum_{i=k+1}^{n}\eta_{i,j}\sum_{i=n\tau+1}^{k}\eta_{i,j}\right|>\frac{n\Delta_T^2}{\sigma^2}
\frac{(\frac kn -\tau)n\tau}{8k}	\right).
\end{align*}

Now,  using lemma \ref{chi2}, we get, since $n\Delta_T^2\ge 64T\sigma^2/\eps$,
\begin{align*}
&\sum_{k\in \{n\tau+n\lambda\Psi_n,\ldots, n-2\}}P\left(\frac1{\sigma^2}\sum_{j=1}^T\left(\sum_{i=n\tau+1}^{k}\eta_{i,j}\right)^2>\frac{n\Delta_T^2}{\sigma^2}
\frac{(\frac kn -\tau)n\tau}{8}	\right)
\\
&
\quad\le
\sum_{k\in \{n\tau+n\lambda\Psi_n,\ldots, n-2\}}P\left(\frac1{\sigma^2}\sum_{j=1}^T\frac{\left(\sum_{i=n\tau+1}^{k}\eta_{i,j}\right)^2}{k-n\tau}>\frac{n\Delta_T^2}{\sigma^2}
\frac{\tau}{8}	\right)
\\
&\quad
\le n\exp\left(-\frac{n\Delta_T^2\tau}{64\sigma^2}
\right).
\end{align*}
As well, since $n\Delta_T^2\ge 32T\sigma^2/\eps^2$,

\begin{align*}
&\sum_{k\in \{n\tau+\lambda\Psi_n,\ldots, n-2\}}P\left(\frac1{\sigma^2}\sum_{j=1}^T\left(\sum_{i=k+1}^{n}\eta_{i,j}\right)^2\frac{k-n\tau}{(n-k)(n-n\tau)}>\frac{n\Delta_T^2}{\sigma^2}
\frac{(\frac kn -\tau)n\tau}{8k}\right)
\\
&\quad\le\sum_{k\in \{n\tau+\lambda\Psi_n,\ldots, n-2\}}P\left(\frac1{\sigma^2}\sum_{j=1}^T\left(\sum_{i=k+1}^{n}\eta_{i,j}\right)^2\frac{1}{(n-k)}>\frac{n\Delta_T^2}{\sigma^2}
\frac{(1 -\tau)\tau}{8}	\right)
\\
&\quad
\le n\exp\left(-\frac{n\Delta_T^2\tau(1-\tau)}{64\sigma^2}
\right).
\end{align*}
Now, let us denote by $\F$, the $\sigma-$algebra spanned by the variables $\{\eta_{i,j},\; i\le n\tau,\; j\le T\}$.
\begin{align*}
&\sum_{k\in \{n\tau+n\lambda\Psi_n,\ldots, n-2\}}P\left(
\bigg|\sum_{j=1}^T\frac2k\sum_{i=1}^{n\tau}\eta_{i,j}\sum_{i=n\tau+1}^{k}\eta_{i,j}\bigg|>n\Delta_T^2
\frac{(\frac kn -\tau)n\tau}{8k}	\right)
\\
&=	\sum_{k\in \{n\tau+n\lambda\Psi_n,\ldots, n-2\}}E\left[P\bigg(\frac1{\sigma^2}
\bigg|\sum_{j=1}^T\sum_{i=1}^{n\tau}\eta_{i,j}\sum_{i=n\tau+1}^{k}\eta_{i,j}\bigg|>\frac{n\Delta_T^2}{\sigma^2}
\frac{(\frac kn -\tau)n\tau}{16}\Big|\F	\bigg)\right]
\\
&\le 2	\sum_{k\in \{n\tau+n\lambda\Psi_n,\ldots, n-2\}}E\left[\exp\left(-\frac{(n\Delta_T^2(\frac kn -\tau)n\tau)^2}{16^2}\frac1{2\sigma^2(k-n\tau)\sum_{j=1}^T(\sum_{i=1}^{n\tau}\eta_{i,j})^2}\right)\right]
\\
&\le 2	\sum_{k\in \{n\tau+n\lambda\Psi_n,\ldots, n-2\}}E\Bigg[\exp\left(-\frac{(n\Delta_T^2(\frac kn -\tau)n\tau)^2}{16^2}\frac1{2\sigma^2(k-n\tau)\sum_{j=1}^T(\sum_{i=1}^{n\tau}\eta_{i,j})^2}\right)\Bigg.	\\\Bigg.&\kern9cm\times\I_{\left\{\sum_{j=1}^T\frac{\left(\sum_{i=1}^{n\tau}\eta_{i,j}\right)^2}{n\tau\sigma^2}\le \frac{n\Delta_T^2\tau}{8\sigma^2}\right\}}\Bigg]
\\
&+	2nP\left(\sum_{j=1}^T\frac{\left(\sum_{i=1}^{n\tau}\eta_{i,j}\right)^2}{n\tau\sigma^2}\ge \frac{n\Delta_T^2\tau}{8\sigma^2}\right)
\\
&\le 2n\exp\left(-\frac{\lambda}{64}\right)
+ 2n\exp\left(-\frac{n\Delta_T^2\tau}{64\sigma^2}
\right).
\end{align*}

To end the proof we investigate the last term: let now $\F_k$ be the $\sigma-$algebra spanned by the variables $\{\eta_{i,j},\; i> k,\; j\le T\}$. We write:

\begin{align*}
&\sum_{k\in \{n\tau+n\lambda\Psi_n,\ldots, n-2\}}  P\left(
\bigg|\sum_{j=1}^T\frac2{n-n\tau}\sum_{i=k+1}^{n}\eta_{i,j}\sum_{i=n\tau+1}^{k}\eta_{i,j}\bigg|>{n\Delta_T^2}
\frac{(\frac kn -\tau)n\tau}{8k}	\right)
\\
&
\le \sum_{k\in \{n\tau+n\lambda\Psi_n,\ldots, n-2\}}E\left[P\left(
\bigg|\sum_{j=1}^T\frac 1{n-n\tau}\sum_{i=k+1}^{n}\eta_{i,j}\sum_{i=n\tau+1}^{k}\eta_{i,j}\bigg|>{n\Delta_T^2}
\frac{(\frac kn -\tau)n\tau}{16k}\Big|\F_k	\right)\right]
\\
&\le 2	\sum_{k\in \{n\tau+n\lambda\Psi_n,\ldots, n-2\}}E\Bigg[\exp\left(-\frac{(n\Delta_T^2(\frac kn -\tau)(n-n\tau)n\tau)^2}{16^2k^2}\frac1{2\sigma^2(k-n\tau)\sum_{j=1}^T(\sum_{i=k+1}^{n}\eta_{i,j})^2}\right)	
\Bigg.
\\\Bigg.
&\kern9cm
\times\I_{\left\{\sum_{j=1}^T\frac{\left(\sum_{i=k+1}^{n}\eta_{i,j}\right)^2}{\sigma^2(n-k)}\le \frac{n\Delta_T^2\tau}{8\sigma^2}\right\}}\Bigg]
\\&\quad+	2nP\left(\sum_{j=1}^T\frac{\left(\sum_{i=k+1}^{n}\eta_{i,j}\right)^2}{\sigma^2(n-k)}\ge \frac{n\Delta_T^2\tau}{8\sigma^2}\right)
\\
&\le 2n\exp\left(-\frac{\lambda\tau(1-\tau)^2}{64}\right)
+ 2n\exp\left(-\frac{n\Delta_T^2\tau}{64\sigma^2}
\right).
\end{align*}
Summarizing the elements above, we get:
\begin{equation*}
P\left(|\hat \tau-\tau|\ge \lambda \Psi_n\right)
\le2\left[(16n+2)\exp\left(-\frac{\lambda\eps^4}{16^3}\right)
+12n\exp\left(-\frac{n\Delta_T^2\eps^2}{64\sigma^2}\right)\right].\label{final-beta}
\end{equation*}
Taking now $\lambda=\kappa(\gamma, \eps)\ln(n)$, this proves Proposition \ref{beta}.

 \subsection{Proof of Theorem \ref{hatT}}

Let us in this section define $T_s:= \left(\frac{\sigma^2\ln (d\vee n)}{n}\right)^{\frac{-1}{1+2s}}$.
The following lemma is essential in the sequel.

\begin{lemma}\label{hatTaux}

In  the model \eref{eq:model}, we assume that  $\theta^+$ and $\theta^-$ belong to $\Theta(s,L)$. We suppose that there exists a constant $\alpha>0$ such that
 $$\frac n {\sigma^{2}}\ge \alpha \ln d.$$
Then, for any $\gamma$, if $C_{\mathcal L}$ is large enough (see condition \eref{cond-c} below), there exists a constant $R=R(\gamma,L,C_{\mathcal L},\eps)$  (see condition \eref{cond-R}  such that, if
\begin{equation}
\Delta^2\ge  R\left(\frac{\sigma^2\ln (d\vee n)}{n}\right)^{\frac{2s}{1+2s}}, 
\label{cond(s)}
\end{equation}
then,   as soon as  $n$ is greater than an absolute constant, we have
\begin{equation*}\label{delta-hat-T}
P \left(\left\{\Delta_{\hat T}^2\ge \frac{\Delta^2}2\right\}\cap \{\hat T\le T_s\}\right)\ge 1-n^{-\gamma}.
\end{equation*}
\end{lemma}

The proof is based on an intermediate lemma, stating that, with large probability, $\hat T\le T_s$.
\begin{lemma}\label{hat-tout-seul}
	Under the conditions above, 
	for any $\gamma$, if we have
	\begin{equation}C_{\mathcal L}\ge 16\vee 4L^{2} \vee \frac{4(\gamma+2)}{\alpha^{\frac{1}{1+2s}}}\vee \frac{8L(\gamma+2)^{1/2}}{\alpha^{\frac{1}{2(1+2s)}}},\label{cond-c}\end{equation}  then
	\begin{equation*}
	P(\hat T>T_s)\le (d\vee n)^{-\gamma}\label{hat-T-ks}.
	\end{equation*}
\end{lemma}

To prove the result, we will need a tail bound given in Lemma \ref{GD} below. Recall that $Z$ is defined by $$Z_j= \frac 1n \sum_{i=1}^nY^j_i- \frac 2n \sum_{i=1}^{n/2}Y^j_i,\quad j=1,\ldots, d,$$
that is
$Z_j= \beta_j+\eps_j$, $j=1,\ldots,d,$ where $\beta_j= (1-\tau)(\theta_j^+-\theta_j^-)\I_{\{\tau\ge 1/2\}}+ \tau(\theta_j^+-\theta_j^-) \I_{\{\tau< 1/2\}}$ and $\eps_j\sim\Nr(0,\frac{\sigma^2}{n})$.

\begin{lemma} \label{GD} For $T_s\le \ell\le k
$,
\begin{equation*}
P\left(\Big|\sum_{j=\ell}^k (Z_j)^2-\sum_{j=\ell}^k(\beta_j)^{2}\Big|>x\right)\le 2 \exp \left(-\frac{x^{2}}{64L^{2}T_s^{-2s}\frac{\sigma^2}{n}}\right)+
\exp\left( -\frac{nx}{16\sigma^{2}}\right),
\end{equation*}
as soon as $x\ge 8(\ell-k)\sigma^2/n$.
\end{lemma}

\begin{proof}[Proof of Lemma \ref{GD}]

\begin{align*}
P\left(\Big|\sum_{j=\ell}^k (Z_j)^2-\sum_{j=\ell}^k(\beta_j)^{2}\Big|>x\right)&\le P\left(\sum_{j=\ell}^k (\eps_j)^2+2\Big|\sum_{j=\ell}^k\eps_j\beta_j\Big|
>x\right)
\\
&\le P\left(\sum_{j=\ell}^k (\eps_j)^2>x/2\right)+P\left(\Big|\sum_{j=\ell}^k\eps_j\beta_j\Big|
>x/4\right).
\end{align*}
Now, observe that $\sum_{j=\ell}^k\eps_j\beta_j$ follows a Gaussian distribution $\Nr\big(0, \frac{\sigma^2}{n}\sum_{j=\ell}^k(\beta_j)^2\big)$, so that, using the concentration of the Gaussian distribution (see \eref{Gaussconc}) and the fact that $\sum_{j=\ell}^k(\beta_j)^2\le 2 L^{2}T_s^{-2s}$, since $\theta^+$ and $\theta^-$ are in $\Theta(s,L)$,  we obtain
\begin{align*}\label{CS}
P\left(\Big|\sum_{j=\ell}^k\eps_j\beta_j\Big|
>x/4\right)\le 2\exp \left(-\frac{x^{2}}{64L^{2}T_s^{-2s}\frac{\sigma^2}{n}}\right).
\end{align*}
Now, using \eref{chi2}, we get
\begin{equation*}
P\left(\sum_{j=\ell}^k (\eps_j)^2>x/2\right)\le \exp\left( -\frac{nx}{16\sigma^{2}}\right)\label{CSchi2},
\end{equation*}
as soon as $x\ge 8(k-\ell)\sigma^{2}/n$.
\end{proof}

\begin{proof}[Proof of Lemma \ref{hat-tout-seul}]
	We have
	\begin{equation*}
	P(\hat T>T_s)\le P\left(\exists k\ge \ell\ge T_s,\sum_{j=\ell}^k (Z_j)^2>C_{\mathcal L} k\frac{\sigma^2}{n}\ln (d\vee n)\right).
	\end{equation*}
	
	Now,
	since $k \ge T_s$,
	$$
	\sum_{j=\ell}^k (\beta_j)^{2}\le 2L^{2}T_s^{-2s}=2L^{2}\left(\frac{\sigma^2\ln (d\vee n)}{n}\right)^{\frac{2s}{1+2s}}= 2L^{2} T_s\frac{\sigma^2\ln (d\vee n)}{n}.
		$$
	Thus, if $C_{\mathcal L}\ge 4L^{2}$, we have 
	$\sum_{j=\ell}^k (\beta_j)^{2}\le (C_{\mathcal L}/2) k\frac{\sigma^2}{n}\ln (d\vee n)$.
We get, with $2x:= C_{\mathcal L} k 
\frac{\sigma^2}{n}\ln (d\vee n)$, the following inequality:
	\begin{align*}
	P\left(\sum_{j=\ell}^k  (Z_j)^2>2x\right)&\le 
	P\left(\Big|\sum_{j=\ell}^k  (Z_j)^2-\sum_{j=\ell}^k(\beta_j)^{2}\Big|>x\right).
	\end{align*}

	Now, from Lemma \ref{GD}, as soon as $C_{\mathcal L} k \ln (d\vee n)]\ge 16(k-\ell)$ (which is always true for instance if {{$C_{\mathcal L}\ge 16$}}), we have

		\begin{align*}
	P(\hat T>T_s)&\le \sum_{k\ge \ell\ge T_s}P\left(\sum_{j=\ell}^k (Z_j)^2>C_{\mathcal L}k\frac{\sigma^2}{n}\ln (d\vee n)\right)
	\\
	& \le \sum_{k\ge \ell\ge T_s}2\exp\left( -\frac{(C_{\mathcal L} k\frac{\sigma^2}n \ln (d\vee n))^{2}}{64L^{2}T_s^{-2s}\frac{\sigma^2}{n}}\right)+
	\exp \left(-C_{\mathcal L} k\ln (d\vee n)/4\right)
	\\
	&\le d^{2}\left[2\exp\left( -\frac{(C_{\mathcal L}T_s\frac{\sigma^2}n \ln (d\vee n))^{2}}{64L^{2}T_s^{-2s}\frac{\sigma^2}{n}}\right)+
	\exp \left(-C_{\mathcal L} T_s\ln (d\vee n)/4\right)\right]
	\\
	&\le d^{2}\left[2\exp\left( -\frac{C_{\mathcal L}^2}{64L^{2}}\left(\frac{\sigma^{2}}{n}\right)^{-\frac{1}{1+2s}}(\ln (d\vee n))^{\frac{2s}{1+2s}}\right)\right.
\\&\left.
	\quad+\exp
	\left(
	-\frac{C_{\mathcal L}}4
	\left(\frac{\sigma^{2}}{n}\right)^{-\frac{1}{1+2s}}
	(\ln (d\vee n))^{\frac{2s}{1+2s}}
	\right)
\right]
\\
&\le d^{2}\left[2\exp\left( -\frac{C_{\mathcal L}^2}{64L^{2}}\alpha^{\frac{1}{1+2s}}(\ln (d\vee n))\right)
+\exp
\left(
-\frac{C_{\mathcal L}}4
\alpha^{\frac{1}{1+2s}}
(\ln (d\vee n))
\right)
\right].
	\end{align*}

	This last term is clearly less than $3(d\vee n)^{-\gamma}$, as, by assumption,
	$\frac{C_{\mathcal L} \alpha^{\frac{1}{1+2s}}}{4}>\gamma+2$,  as well as $\frac{C_{\mathcal L}^2\alpha^{\frac{1}{1+2s}}}{64L^{2}}>\gamma+2$.
\end{proof}

Equipped with Lemma \ref{hat-tout-seul}, let us go back to the proof of Lemma \ref{hatTaux}.
\begin{proof}[Proof of Lemma 1]

To simplify the exposition, let us suppose that $\tau\ge 1/2$, the other case can be treated similarly, with elementary modifications.

We may write

$$\Delta_{{\hat T}}^{2}=\Delta^{2}-\sum_{j=\hat T+1}^{T_s}(\theta^{+}_j-\theta^{-}_j)^{2}-\sum_{j=T_s+1}^{d}(\theta^{+}_j-\theta^{-}_j)^{2}.$$
Yet, 
$$\sum_{k=T_s+1}^{d}(\theta^{+}_j-\theta^{-}_j)^{2}\le 4L^{2}T_s^{-2s}
=4L^{2}\left(\frac{\sigma^2\ln (d\vee n)}{n}\right)^{\frac{2s}{1+2s}}
\le \Delta^{2}/4,$$
as soon as {{$R\ge 16L^{2}$}}, since $\Delta^2\ge  R\left(\frac{\sigma^2\ln (d\vee n)}{n}\right)^{\frac{2s}{1+2s}}$ (condition \eref{cond(s)}).
Then,
$$
P\left(\{\hat T\le T_s\}\cap \{\Delta_{\hat T}^{2}\le  \Delta^{2}/2\}\right)
\le P\left(\{\hat T\le T_s\}\cap\bigg\{ \sum_{j=\hat T+1}^{T_s}(\theta^{+}_j-\theta^{-}_j)^{2}\ge \Delta^{2}/4\bigg\}\right).
$$
Now,
\begin{align*}
&P\left(
\{\hat T\le T_s\}\cap \bigg\{\sum_{j=\hat T+1}^{T_s}(\theta^{+}_j-\theta^{-}_j)^{2}\ge \Delta^{2}/4\bigg\}
\right)
\\
&\qquad\le P\left(\{\hat T\le T_s\}\cap \bigg\{\sum_{j=\hat T+1}^{T_s}\frac{(Z^{j})^{2}}{(1-\tau)^{2}}+\left[(\theta_{+}^{j}-\theta_{-}^{j})^{2}-\frac{(Z^{j})^{2}}{(1-\tau)^{2}}\right]\ge 
\Delta^{2}/4 \bigg\}\right).
\end{align*}
From the construction of $\hat T$, we know that $$\sum_{j=\hat T+1}^{T_s}\frac{(Z^{j})^{2}}{(1-\tau)^{2}}\le \frac{C_{\mathcal L} T_s}{(1-\tau)^{2}}\frac{\sigma^2}n{\ln (d\vee n)}=\frac{C_{\mathcal L}}{(1-\tau)^{2}}\left(\frac{\sigma^2\ln (d\vee n)}{n}\right)^{\frac{2s}{1+2s}}.$$
Thus, if $C_{\mathcal L} /(1-\tau)^{2}\le R/8$,  $$\sum_{j=\hat T+1}^{T_s}\frac{(Z^{j})^{2}}{(1-\tau)^{2}}\le \Delta^{2}/8.$$
Consequently, 
\begin{align*}
P&\left(\{\hat T\le T_s\}\cap \{\Delta_{\hat T}^{2}\le  \Delta^{2}/2\}\right)
\\&\le P\left(\{\hat T\le T_s\}\cap \bigg\{\sum_{j=\hat T+1}^{T_s}\left[(\theta_{+}^{j}-\theta_{-}^{j})^{2}-\frac{(Z^{j})^{2}}{(1-\tau)^{2}}\right]\ge 
\Delta^{2}/8 \bigg\}\right).
\\
&\le
 \sum_{k=1}^ {T_s} P\left(\{\hat T= k\}\cap \bigg\{\bigg|\sum_{j=k+1}^{T_s}\left[(\theta_{+}^{j}-\theta_{-}^{j})^{2}-\frac{(Z^{j})^{2}}{(1-\tau)^{2}}\right]\bigg|\ge 
 \Delta^{2}/8 \bigg\}\right)
 \\
 &\le
 \sum_{k=1}^ {T_s} P\left(\bigg|\sum_{j=k+1}^{T_s}\left[(\beta_j)^{2}-(Z^{j})^{2}\right]\bigg|\ge 
 (1-\tau)^{2}\Delta^{2}/8 \right)
 \\
 &\le
\sum_{k=1}^ {T_s} P\left(\sum_{j=k+1}^{T_s}(\eps_j)^{2}+2\bigg|\sum_{j=k+1}^{T_s}\eps_j\beta_j\bigg|\ge 
(1-\tau)^{2}\Delta^{2}/8 \right).
\end{align*}

It remains to proceed as for  Lemma \ref{GD}, using the standard inequalities as \eref{Gaussconc} and \eref{chi2}.

As soon as $(1-{\tau})^{2} n\Delta^{2}/(64\sigma^2)\ge T_s$, which is always true if $R\ge \frac{16}{(1-\tau)^2}$, we obtain:
$$\sum_{k=1}^ {T_s} P\left(\sum_{j=k+1}^{T_s}(\eps_j)^{2}\ge (1-{\tau})^{2} \Delta^{2}/16\right)
\le T_s\exp\left(- \frac{(1-{\tau})^{2}n \Delta^{2} }{128\sigma^{2}} \right).
$$

Moreover, since $\mbox{Var}\left(\sum_{j=k+1}^{T_s}\eps_j\beta_j\right)=\frac{\sigma^{2}}{n}\sum_{j=k+1}^{T_s}(\beta_j)^{2}\le
\frac{\sigma^{2}\Delta^{2}(1-\tau)^{2}}{n}$,
$$
\sum_{k=1}^ {T_s} P\left(\bigg|\sum_{j=k+1}^{T_s}\eps_j\beta_j\bigg|\ge (1-{\tau})^{2} \Delta^{2}/32\right)
\le T_s\exp\left( -\frac{(1-{\tau})^{2}n \Delta^{2}}{512\sigma^{2}}\right).
$$

 Now, $\frac{n\Delta^{2}}{\sigma^{2}}\ge {R\ln (d\vee n)T_s}$. Hence, for $R$ large enough, the right-hand terms may be bounded by $ n^{-\gamma}$ for $n$ greater than some absolute constant. Combining this bounds with Lemma \ref{hat-tout-seul}, we get the desired result, 
 as soon as  conditions \eqref{cond-c} and 
 \begin{equation}
 \label{cond-R}
 R\ge 16L^{2}\vee \frac{8C_{\mathcal L}}{\eps^2}\vee \frac{64}{\eps^2}\vee \frac{2^9\gamma}{\eps^2}
 \end{equation}
 are satisfied.
\end{proof}

\begin{proof}[Proof of Theorem \ref{hatT}]

To end the proof of the theorem, 
 
we use Lemma \ref{hatTaux}, the definition of $T_s$, and Proposition \ref{beta}, for any $\gamma,\gamma'$

\begin{align*}
&P\left(|\hat \tau(\hat T)-\tau|\ge \kappa(\gamma,\eps)\frac{\sigma^2\ln(n)}{n\Delta^2}\right)
\\&\leq P\left(\left\{|\hat \tau(\hat T)-\tau|\ge \kappa(\gamma,\eps)\frac{\sigma^2\ln(n) }{n\Delta^2}\right\}\cap\left\{\Delta_{\hat T}^2\ge \Delta^2/2\right\}\cap\{\hat T\le T_s\}\right)+n^{-\gamma}
\\&\leq \sum_{T=1}^{T_s} P\left(\left\{|\hat \tau(\hat T)-\tau|\ge \kappa(\gamma,\eps)\frac{\sigma^2\ln(n)}{n\Delta^2}\right\}\cap\left\{\Delta_{\hat T}^2\ge  \Delta^2/2\right\}\cap\{\hat T=T\}\right)+n^{-\gamma}
\\&\leq \sum_{T=1}^{T_s} P\left(\left\{|\hat \tau( T)-\tau|\ge \kappa(\gamma,\eps)\frac{\sigma^2\ln(n) }{2n\Delta_T^2}\right\}\cap\left\{\Delta_{T}^2\ge \Delta^2/2\right\}\cap\{\hat T=T\}\right)+n^{-\gamma}
\\&\leq \sum_{T=1}^{T_s} P\left(\left\{|\hat \tau( T)-\tau|\ge \kappa(\gamma,\eps)\frac{\sigma^2\ln(n) }{2n\Delta^2}\right\}\cap\left\{\Delta_{T}^2\ge  \frac{R T\sigma^2\ln (d\vee n)}{2n}\right\}\cap\{\hat T=T\}\right)\\&\kern12cm+n^{-\gamma}
\\&\leq T_sn^{-\gamma'}+n^{-\gamma},
\end{align*}

as soon as $R$ is large enough, which proves the theorem.
\end{proof}

\section{Appendix: concentration inequalities\label{conc-ineq}}

\paragraph{Simple Gaussian concentration}
If  $N\sim \Nr(0,1)$, then it is well known that, for $x>0$,
\begin{align} P(|N|>x)\le 2\exp \Big(-\frac{x^2}2\Big).
\label{Gaussconc}
\end{align}
\paragraph{Concentration for the Chi-square distribution (large deviations)}
\begin{lemma}\label{chi2}
Let $k$ be a positive integer and $U$ be a $\chi_k^2$ variable. Then
 $$\forall u^2\ge 4k,\quad P( U\ge u^2)\le \exp \left(\frac{-u^2}{8}\right).$$
\end{lemma}
This lemma is standard. We give a sketch of proof, for sake of simplicity.
 Recall the following result given for instance in \cite{MassStF}. If $X_t$ is a centered Gaussian process such that $\sigma^2:=\sup_t E X_t^2$, then
\begin{equation*}
\forall y>0,\quad P\left(\sup_t X_t- E \sup_t X_t\ge y\right)\le
\exp \left(-\frac {y^2}{2\sigma^2}\right).\label{supgaus}
\end{equation*}
Let $Z_1,\ldots,Z_k$ i.i.d. standard Gaussian variables such that
\begin{align*}
P( U\ge u^2)&=P\left(\sum_{i=1}^kZ_i^2\ge  u^2\right)
= P\left(\sup_{a\in S_1} \sum_{i=1}^ka_iZ_i\ge (u^2)^{1/2}\right)\\
&= P\left(\sup_{a\in S_1} \sum_{i=1}^ka_iZ_i-E\sup_{a\in S_1}
\sum_{i=1}^ka_iZ_i\ge (u^2)^{1/2}-E\sup_{a\in S_1}
\sum_{i=1}^ka_iZ_i\right)
\end{align*}
where $S_1=\{a\in R^k, \|a_i\|_{l^2(k)}=1\}$. Denote
$$X_a=\sum_{i=1}^ka_iZ_i\quad\mbox{ and }\quad y=(u^2)^{1/2}-E\left[\sup_{a\in S_1} \sum_{i=1}^ka_iZ_i\right].$$
Notice that
$$a\in S_1 \Rightarrow E\left(X_a\right)^2=1,$$  as well as
$$E\sup_{a\in S_1}X_a=E\left[\sum_{i=1}^kZ_i^2\right]^{1/2}\le \left[E\sum_{i=1}^kZ_i^2\right]^{1/2}=k^{1/2}. $$
Since $u^2\ge 4 k$, the announced result is proved as soon as
$y> (u^2)^{1/2}/2$.

\paragraph{Concentration for the Chi-square distribution (moderate deviations)}
\begin{lemma}\label{chi-deux-MD}
	If $Z$ has a $\chi^2 $ distribution with $k$ degrees of freedom, then for any $z>0$,
	\begin{align*}
	P(Z-k>z)\le \exp\left({-\frac{z^2}{16k}}\right).
	\end{align*}
	
\end{lemma}

\begin{proof}
	For all $0<t<\frac12$,
\begin{align*}
P(Z-k>z)&\le \exp{(-(k+z)t)}E[\exp{(tZ)}]
\\
&\le \exp\left(-(k+z)t-\frac k2\ln(1-2t)\right),
\end{align*}since the moment generating function of a $\chi^2 $ distribution with $k$ degrees of freedom is $t\mapsto (1-2t)^{k/2}$, defined for $0<t<\frac12$.
Taking $t=\frac12(1-(1+\frac z k)^{-1})$ and using $\ln(1+u)\le u-\frac{u^2}8$, for $0\le u\le 1$, we get the result.
\end{proof}

\section*{Acknowlegdement}

We warmly thank Oleg Lepski for fruitful discussions about minimax rates in change-point detection, which helped to significantly improve a first version of the paper. We are also grateful to Sylvain Delattre for his interesting comments about the simulation study.

\bigskip

This research has been partly supported by French National Research Agency (ANR) as part of the project FOREWER ANR-14-CE05-0028.

\bibliographystyle{plainnat}
\bibliography{SmoothClustBiblio}

\end{document}